\documentclass[12pt,a4paper]{amsart}
\usepackage{amssymb}
\usepackage{calrsfs}
\usepackage{url}

\newcommand{\F}{{\mathbb{F}}}
\newcommand{\C}{{\mathbb{C}}}
\newcommand{\Z}{{\mathbb{Z}}}

\newcommand{\bB}{{\mathbf{B}}}
\newcommand{\bU}{{\mathbf{U}}}
\newcommand{\bC}{{\mathbf{C}}}
\newcommand{\bG}{{\mathbf{G}}}
\newcommand{\bH}{{\mathbf{H}}}
\newcommand{\bK}{{\mathbf{K}}}
\newcommand{\bL}{{\mathbf{L}}}
\newcommand{\bP}{{\mathbf{P}}}
\newcommand{\bT}{{\mathbf{T}}}
\newcommand{\bW}{{\mathbf{W}}}

\newcommand{\fX}{{\mathfrak{X}}}
\newcommand{\cE}{{\mathcal{E}}}

\newcommand{\cN}{{\mathcal{N}}}
\newcommand{\cO}{{\mathcal{O}}}
\newcommand{\cH}{{\mathcal{H}}}

\newcommand{\tis}{{\tilde{s}}}

\newcommand{\tbG}{{\tilde{\bG}}}

\newcommand{\tF}{{\tilde{F}}}

\newcommand\CF{\operatorname{CF}}
\newcommand\Irr{\operatorname{Irr}}

\makeatletter
\let\@wraptoccontribs\wraptoccontribs
\makeatother

\newtheorem{thm}{Theorem}[section]
\newtheorem{cor}[thm]{Corollary}
\newtheorem{prop}[thm]{Proposition}
\newtheorem{lem}[thm]{Lemma}

\theoremstyle{definition}

\newtheorem{exmp}[thm]{Example}
\newtheorem{abs}[thm]{}

\theoremstyle{remark}
\newtheorem{rem}[thm]{Remark}

\renewcommand{\leq}{\leqslant}
\renewcommand{\geq}{\geqslant}

\begin{document}
\title[Character tables of $E_6(q)_{\text{ad}}$ and
${^2\!E}_6(q)_{\text{ad}}$]{On the character tables of the finite 
reductive groups $E_6(q)_{\text{ad}}$ and ${^2\!E}_6(q)_{\text{ad}}$}
\author{Meinolf Geck}
\address{Lehrstuhl f\"ur Algebra, FB Mathematik\\ Universit\"at Stuttgart\\ 
Pfaffenwaldring 57\\ 70569 Stuttgart, Germany}
\email{meinolf.geck@mathematik.uni-stuttgart.de}
\contrib{With an appendix by J. Hetz}

\date{May 11, 2024}

\begin{abstract} We show how the character tables of the groups 
$E_6(q)_{\text{ad}}$ and ${^2\!E}_6(q)_{\text{ad}}$ can be constructed, 
where $q$ is a power of~$2$. (Partial results are also obtained for
any $q$ not divisible by~$3$.) This is based on previous work by Hetz, 
Lusztig, Malle, Mizuno and Shoji, plus computations using Michel's 
version of {\sf CHEVIE}. We also need some general results that are 
specific to semisimple groups which are not of simply connected type.
A further crucial ingredient is the determination of the values of the 
unipotent characters on unipotent elements for groups of type $D_4$ and 
$D_5$ (in characteristic~$2$).
\end{abstract}

\maketitle

\section{Introduction} \label{sec0}

Let $p$ be a prime and $k=\overline{\F}_p$ be an algebraic closure of
the field with $p$ elements. Let~$\bG$ be a connected reductive algebraic
group over $k$ and assume that $\bG$ is defined over the finite subfield
$\F_q\subseteq k$, where $q$ is a power of~$p$. Let $F\colon \bG\rightarrow 
\bG$ be the corresponding Frobenius map and $\bG^F=\bG(\F_q)$ be the
finite group of fixed points. This paper is part of an ongoing project 
(involving various authors) to obtain as much information as possible 
about the character table of $\bG^F$, where $\bG$ is simple of exceptional
type. (See the recent survey \cite{gema} for general background.) Here, 
we consider the case where $\bG$ is simple, adjoint of type $E_6$. We 
shall obtain complete results for $p=2$, and at least partial results for 
any $p\neq 3$. (For $p=3$, our methods do not yield anything new.) 

The general framework is provided by Lusztig's geometric approach 
\cite{LuB}, \cite{L7}, starting from the generalised characters 
$R_{\bT,\theta}^\bG$ of \cite{DeLu}. Our point of view will be that the 
values of these generalised characters are already known. Since, 
in general, they do not span the whole space of complex-valued class 
functions on $\bG^F$, we then need to focus on determining the ``missing'' 
functions. For this purpose, one proceeds along a list $\Sigma_1,\ldots,
\Sigma_N$ of the $F$-stable conjugacy classes of $\bG$. For each~$i$, the 
set $\Sigma_i^F$ splits into conjugacy classes of $\bG^F$ where the
splitting is controlled by the finite group $A_\bG(g_i):=C_\bG(g_i)/
C_\bG^\circ(g_i)$ ($g_i \in \Sigma_i^F$). The size of $A_\bG(g_i)$ is a 
kind of measure for the difficulty of computing the character values on 
elements in $\Sigma_i^F$. For example, if $A_\bG(g_i)=\{1\}$, then 
$\Sigma_i^F$ is a single $\bG^F$-conjugacy class and the character 
values on $g_i$ can be described in terms of the various $R_{\bT,
\theta}^\bG$. If $A_\bG(g_i) \neq \{1\}$, then this is no longer true 
in general. 

In this set-up, it will be convenient to consider the vector space 
$\CF(\bG^F\mid \Sigma_i)$ of all complex-valued class functions on 
$\bG^F$ that are $0$ on elements
outside~$\Sigma_i^F$. Then the aim is to find a basis of $\CF(\bG^F\mid 
\Sigma_i)$ such that
\begin{itemize}
\item the values of the functions in that basis are explicitly
known and 
\item the decomposition of each function in that basis as 
a linear combination of $\Irr(\bG^F)$ is explicitly known.
\end{itemize}
If this problem is solved for each $\Sigma_i$, then we can construct the
character table of $\bG^F$ from the bases of the various subspaces $\CF
(\bG^F\mid \Sigma_i)$. Now Lusztig \cite[0.4]{L7} presents a general 
strategy for solving that problem (based on results in \cite{L5}), but this 
involves the delicate issue of fixing normalisations of characteristic 
functions of $F$-invariant character sheaves on $\bG$. Similarly to our 
previous work \cite{unif4} on type $F_4$, the special feature of groups of 
adjoint type~$E_6$ in characteristic~$2$ is that one can largely avoid that 
issue, and only deal with it in some very special cases.

However, an additional complication for $\bG$ of adjoint type~$E_6$ arises 
from the fact that there are semisimple elements in $\bG$ whose centraliser 
is not connected. In principle, this could be avoided by working with a
simple group of simply connected type~---~but this would create numerous 
further complications of a different (and arguably more involved) nature, 
caused by the fact that the center is not connected. For example, the main 
results of Shoji \cite{S2}, \cite{S3} (on which we heavily rely) are only
proved for groups with a connected center. Another possibility would be to 
realise our group as $\bG\cong \tbG/Z(\tbG)$ where $\tbG$ has a simply 
connected derived subgroup \textit{and} the center $Z(\tbG)$ is connected. 
However, this would also create additional complications as far as 
characteristic functions of cuspidal character sheaves are concerned (which 
are needed in our argument). Thus, all in all, we have chosen to work 
directly with $\bG$ simple, adjoint of type $E_6$.

In Section~\ref{prelim}, we prepare some basic notation and recall (known)
reduction techniques for the determination of character values. In
Section~\ref{subsp}, we consider first examples concerning the subspaces 
$\mbox{CF}(\bG^F\mid \Sigma_i)$ as above. In particular, the values of 
all unipotent characters on unipotent elements are determined for groups 
$\bG$ of classical type $D_4$ and~$D_5$ in characteristic~$2$. (These results,
which are partly new and perhaps of independent interest, are needed later
on in the discussion of groups of type $E_6$.) Section~\ref{scad} contains 
the main general results of this paper. These provide some tools for 
dealing with the situation where $\bG$ is semisimple but not of simply 
connected type.

For $\bG$ simple, adjoint of type $E_6$ and $p=2$, the values of the
unipotent characters on unipotent elements have already been determined
by Hetz \cite{Het3}, \cite{Het4}. Most of the remaining work centers 
around a particular class $\Sigma_i$ which is the supporting set for the 
cuspidal character sheaves on $\bG$. In this case, we have $A_\bG(g_i)\cong 
\Z/3\Z\times \Z/3\Z$ for $g_i \in \Sigma_i^F$, where the semisimple part of 
$g_i$ has order~$3$ and a non-connected centraliser. In Section~\ref{e6p2} 
we show how to find a basis of $\CF(\bG^F\mid \Sigma_i)$ which 
satisfies the above requirements; this relies on the results in 
Section~\ref{scad} combined with previous results from \cite{S3}, 
\cite{pamq}. (And this part of the argument does not only work for $p=2$ but 
for any $p\neq 3$; on the other hand, if $p=3$, then there is no element
$g_i$ at all with the above properties, so our methods do not yield anything
new for this case.) Once this is achieved, we can basically follow the 
strategy in \cite{unif4} to complete the character table of~$\bG^F$; see 
Section~\ref{finale6}. 

Finally, we point out that we do describe a concrete procedure for 
constructing the character table of $\bG^F$, but we shall not present 
actual tables of character values. With some more effort (mainly of a
computational nature), it should certainly be possible to produce a 
``generic'' table with the values of the unipotent characters on all 
elements of~$\bG^F$. I understand that Jonas Hetz (private communication)
has recently produced such a ``generic'' table for the unipotent 
characters of $F_4(q)$ where $q$ is a power of~$2$, following the strategy 
in \cite{unif4}. However, there are non-unipotent characters whose values 
on certain elements are sums of roots of unity with $|\bW|$ distinct terms, 
where $\bW$ is the Weyl group of $\bG$. For $\bG$ of type $E_6$ we have 
$|\bW|=51840$, it is to be discussed if one really wants to write down 
explicit tables, or even store them on a computer.

\section{Preliminaries} \label{prelim}

Let $\bG$ be a connected reductive group and $F\colon \bG\rightarrow \bG$ 
be a Frobenius map, with respect to an $\F_q$-rational structure of $\bG$. 
Let $\CF(\bG^F)$ be the vector space of complex-valued class functions 
on~$\bG^F$, with standard inner product $\langle \;,\; \rangle_{\bG^F}$ 
given by 
\[ \langle f,f'\rangle_{\bG^F}:=|\bG^F|^{-1}\sum_{g \in \bG^F} f(g)
\overline{f'(g)} \qquad \mbox{for $f,f'\in \CF(\bG^F)$},\]
where the bar denotes complex conjugation. The set $\Irr(\bG^F)$ of 
irreducible characters of~$\bG^F$ forms an orthonormal basis of 
$\CF(\bG^F)$. For any subset $C\subseteq \bG^F$ that is a union of 
$\bG^F$-conjugacy classes, we denote by $\varepsilon_C \in \CF
(\bG^F)$ the indicator function of $C$, that is, for $x\in \bG^F$ we have 
$\varepsilon_C(x)=1$ if $x\in C$, and $\varepsilon_C(x)=0$ otherwise. Note
that, if $C$ is a single $\bG^F$-conjugacy class and $g\in C$, then 
\[ \rho(g)=|C_\bG(g)^F|\langle \rho,\varepsilon_C\rangle_{\bG^F}
\qquad\mbox{for all $\rho\in \Irr(\bG^F)$}.\]
Let $\fX(\bG,F)$ be the set of all pairs $(\bT,\theta)$ where $\bT
\subseteq \bG$ is an $F$-stable maximal torus and $\theta \in \Irr(\bT^F)$.
For such a pair $(\bT,\theta)$ let $R_{\bT,\theta}^\bG\in \CF(\bG^F)$ be 
the generalised character defined by Deligne and Lusztig \cite{DeLu}. It 
is known that every $\rho\in \Irr(\bG^F)$ occurs in $R_{\bT,\theta}^\bG$
for some $(\bT,\theta)\in \fX(\bG,F)$. Of particular interest is the case 
where this happens for $\theta=1$ (the trivial character): in that case, 
$\rho$ is called a ``unipotent'' character. We refer to \cite{rDiMi},
\cite{gema} for general properties of the generalised characters 
$R_{\bT,\theta}^\bG$.

\begin{abs} \label{unif} Let $f\in \CF(\bG^F)$. Then $f$ is called
a ``uniform'' function if $f$ can be written as a linear combination of 
$R_{\bT, \theta}^\bG$ for various $(\bT,\theta)\in \fX(\bG,F)$. In this 
case, we have (see \cite[Prop.~10.2.4]{rDiMi}): 
\[ f=|\bG^F|^{-1}\sum_{(\bT,\theta)\in \fX(\bG,F)} |\bT^F| \langle 
f,R_{\bT,\theta}^\bG\rangle_{\bG^F}R_{\bT,\theta}^\bG.\]
For example, if $\Sigma$ is an $F$-stable conjugacy class of $\bG$, then
$\Sigma^F$ is a union of $\bG^F$-conjugacy classes and the indicator 
function $\varepsilon_{\Sigma^F}\in \CF(\bG^F)$ is a uniform function; 
see \cite[Appendix]{mylaus} (and also \cite{Lu21} for a different proof).

The subspace of $\CF(\bG^F)$ consisting of all uniform functions provides 
a first, and very useful approximation to $\CF(\bG^F)$ (especially
for $\bG$ of low rank, $\leq 8$ say).
\end{abs}

\begin{exmp} \label{unif1} Let $C$ be a $\bG^F$-conjugacy class and 
$g\in C$. Assume that the indicator function $\varepsilon_C\in \CF(\bG^F)$ 
is a uniform function. Let $\rho\in \Irr(\bG^F)$. Then, expressing 
$\varepsilon_C$ as in~\ref{unif}, we obtain the following formula:
\[ \rho(g)=|C_\bG(g)^F|\langle \rho,\varepsilon_C\rangle_{\bG^F}=
|\bG^F|^{-1}\sum_{(\bT,\theta)\in \fX(\bG,F)} |\bT^F|\,\langle 
R_{\bT,\theta}^\bG,\rho\rangle_{\bG^F}\,R_{\bT,\theta^{-1}}^\bG(g).\]
Hence, in this case, the value $\rho(g)$ is determined by the multiplicities
$\langle R_{\bT,\theta}^\bG,\rho \rangle_{\bG^F}$ and the values 
$R_{\bT,\theta}^\bG(g)$, where $(\bT,\theta)$ runs over all pairs in 
$\fX(\bG,F)$. 

This situation occurs, for example, when the center $Z(\bG)$ is connected 
and the root system of $\bG$ is a direct sum of root systems of type $A_m$, 
for various $m\geq 1$. In that case, \textit{every} class function on $\bG^F$ 
is uniform; see \cite[Cor.~2.4.19]{gema} (and the references there).
\end{exmp}

\begin{exmp} \label{pconst} A class function $\gamma\in \mbox{CF}(\bG^F)$ is
called $p$-constant if, for all $g\in \bG^F$, we have $\gamma(g)=\gamma(s)$ 
where $s\in \bG^F$ is the semisimple part of~$g$. In this case, we have:
\[ f \in \mbox{CF}(\bG^F) \text{ uniform} \qquad \Rightarrow \qquad
\gamma{\cdot} f\in \mbox{CF}(\bG^F) \text{ uniform}.\]
(This immediately follows from \cite[Prop.~10.1.6]{rDiMi}.) Let $(\bT,
\theta)\in \fX(\bG,F)$ and $u\in \bG^F$ be unipotent. Then the value
$R_{\bT,\theta}^\bG(u)$ does not depend on $\theta$, and will be denoted
by $Q_\bT^\bG(u)$. The function $u \mapsto Q_{\bT}^\bG(u)$ is called
the Green function associated with~$\bT$. We may regard $Q_{\bT}^\bG$ as 
a function on all of $\bG^F$ by letting its value be~$0$ on non-unipotent
elements. With this convention, we have $Q_\bT^\bG= \gamma_{\text{uni}}^\bG
{\cdot} R_{\bT,\theta}^\bG$ where $\gamma_{\text{uni}}^\bG\in \mbox{CF}
(\bG^F)$ is the class function which takes value~$1$ on unipotent
elements, and value~$0$ otherwise. Then $\gamma_{\text{uni}}^\bG$ is 
$p$-constant and so $Q_\bT^\bG$ (as a function on all of $\bG^F$) is 
uniform.
\end{exmp}
 
\begin{abs} \label{exp01} An $F$-stable closed subgroup $\bL\subseteq \bG$
is called a ``regular subgroup'' if $\bL$ is a Levi complement in some (not
necessarily $F$-stable) parabolic subgroup $\bP\subseteq \bG$. These 
subgroups can be used for inductive arguments, in the following situation.
Let $g\in \bG$ and $g=su=us$ be the Jordan decomposition of $g$, where 
$s\in \bG$ is semisimple and $u\in \bG$ is unipotent. We say that $g$ (or 
its conjugacy class) is ``isolated'' if $C_{\bG}^\circ(s)$ is not contained 
in a Levi subgroup of a proper parabolic subgroup of~$\bG$. Assume now that 
$g\in \bG^F$ is not isolated. Then there will be a regular subgroup $\bL
\subsetneqq \bG$ such that $C_\bG^\circ(s)\subseteq \bL$. In this case, 
the values $\rho(g)$, for any $\rho\in \Irr(\bG^F)$, can be computed using 
Schewe's formula 
\[ \rho(g)=\sum_{\psi\in \Irr(\bL^F)} \langle R_{\bL\subseteq \bP}^\bG
(\psi),\rho\rangle_{\bG^F} \psi(g) \qquad\mbox{(see 
\cite[Cor.~3.3.13]{gema})}.\]
Here, $R_{\bL\subseteq \bP}^\bG$ is Lusztig's ``twisted induction'';
see \cite[\S 9.1]{rDiMi}. In order to evaluate the above formula,
we need to assume, firstly, that the character table of the smaller 
group $\bL^F\subsetneqq \bG^F$ is known~---~which we can do by an 
inductive procedure. Secondly, we need to know the decomposition of 
$R_{\bL\subseteq \bP}^\bG(\psi)$, for any $\psi\in \Irr(\bL^F)$, as a 
linear combination of $\Irr(\bG^F)$. For the groups that we are going to
consider here, the latter problem is under control as discussed in
\cite[\S 4.6, \S 4.7]{gema}. If $\psi$ is a unipotent character, then
the solution is implemented by the function {\tt LusztigInductionTable} 
in Michel's {\sf CHEVIE} \cite{gap3jm}.
\end{abs}

\begin{abs} \label{rem00} Let $\Sigma$ be an $F$-stable conjugacy class
of $\bG$. We recall how the splitting of $\Sigma^F$ into conjugacy classes
of $\bG^F$ is determined. Let us fix an element $g_0\in \Sigma^F$ and set 
$A_\bG(g_0):=C_\bG(g_0)/C_\bG^\circ(g_0)$. Then $F$ induces an automorphism 
of $A_\bG(g_0)$ that we denote by the same symbol. Given $a\in A_\bG(g_0)$, 
let $\dot{a}\in C_\bG(g_0)$ be a representative of~$a$. By Lang's Theorem, 
we can write $\dot{a}=x^{-1}F(x)$ for some $x\in \bG$; then $g_a :=xgx^{-1}
\in\Sigma^F$. Let $C_a$ be the $\bG^F$-conjugacy class of~$g_a$. One easily 
checks that $C_a$ only depends on~$a$ (and not on the choice of $x$); 
furthermore $\Sigma^F=\bigcup_{a\in A_\bG(g_0)} C_a$, where $C_a=C_{a'}$ 
if and only if $a,a'$ are $F$-conjugate in $A_\bG(g_0)$, that is, there 
exists some $b\in A_\bG(g_0)$ such that $a'=baF(b)^{-1}$. (See 
\cite[\S 2.7]{gema} for further details.) In particular, if $C_\bG(g_0)$
is connected, then $\Sigma^F$ is a single $\bG^F$-conjugacy class.
\end{abs}

\begin{exmp} \label{exp00} Let $\Sigma$ be an $F$-stable conjugacy class
of $\bG$. Let $g_0\in \Sigma^F$ and assume that the group $A_\bG(g_0)$ is 
abelian. Then the map $\varphi\colon A_\bG(g_0)\rightarrow A_\bG(g_0)$, 
$a\mapsto a^{-1}F(a)$, is a group homomorphism with kernel equal to 
$A_\bG(g_0)^F$; the $F$-conjugacy classes of $A_\bG(g_0)$ are precisely 
the cosets of $A_\bG(g_0)$ modulo the image of $\varphi$. Thus, 
we conclude:
\begin{itemize}
\item[(a)] $\;\,|A_\bG(g_0)^F|=$ number of $F$-conjugacy classes of 
$A_\bG(g_0)$.
\item[(b)] $\,$ If $A_\bG(g_0)^F=\{1\}$, then $C:=\Sigma^F$ is a single
conjugacy class of $\bG^F$.
\end{itemize}
In case (b), we can apply the remarks in~\ref{unif}; then the indicator 
function $\varepsilon_C\in \CF(\bG^F)$ is a uniform function and so the 
values $\rho(g)$ for any $\rho\in \Irr(\bG^F)$ can be determined using the 
formula in Example~\ref{unif1}.
\end{exmp} 

\begin{abs} \label{shintani} Let $\Sigma$ be an $F$-stable conjugacy class
of $\bG$. There is a natural operation on the $\bG^F$-conjugacy classes 
contained in $\Sigma^F$, defined as follows. Let $g_0\in \Sigma^F$ be fixed. 
If $C$ is a $\bG^F$-conjugacy class contained in $\Sigma^F$, then $C=C_a$ 
where $a\in A_\bG(g_0)$. (Notation as in~\ref{rem00}.) Let $g\in C$ 
and write $g=x^{-1}F(x)$ where $x\in \bG$. Then one easily checks that $g':=
F(x)x^{-1}=xgx^{-1}\in \bG^F$ and that the $\bG^F$-conjugacy class of $g'$ 
does not depend on the choice of~$g$ or~$x$; we denote the $\bG^F$-conjugacy 
class of $g'$ by $t_1(C)$; clearly, we have $t_1(C)\subseteq \Sigma^F$. The
map $C\mapsto t_1(C)$ is called ``twisting operator''; see, e.g., Shoji 
\cite[1.16]{S2} or Digne--Michel \cite[Chap.~IV]{DiMi0}. One easily shows 
that 
\[ t_1(C_a)=C_{\bar{g}_0a} \qquad \mbox{for all $a\in A_\bG(g_0)$},\]
where $\bar{g}_0$ denotes the image of $g_0$ in $A_\bG(g_0)$.
We shall also need the following result.
\end{abs}

\begin{lem} \label{exp00a} In the setting of \ref{shintani}, let $a\in 
A_\bG(g_0)$. If $g_1\in C_a$ and $g_2\in t_1(C_a)=C_{\bar{g}_0a}$, then 
the semisimple parts of $g_1,g_2$ are conjugate in $\bG^F$.
\end{lem}

\begin{proof} Let $\dot{a}\in C_\bG(g_0)$ be a representative of $a$ and
write $\dot{a}=y^{-1}F(y)$ where $y\in \bG$. Then $g_a:=yg_0y^{-1} \in 
C_a$; so we may assume that $g_1=g_a$. Now write $g_a=s_au_a=u_as_a$ 
where $s_a\in \bG^F$ is semisimple and $u_a\in \bG^F$ is unipotent. 
Then $g_a\in C_\bG^\circ(s_a)$; see \cite[Prop.~3.5.3]{rDiMi}. By Lang's
Theorem applied to $C_\bG^\circ(s_a)$, we can find some $x\in 
C_\bG^\circ(s_a)$ such that $g_a=x^{-1}F(x)$. By \ref{shintani}, we have 
$g':=xg_ax^{-1}\in C_{\bar{g}_0a}$; so we may assume that $g_2=g'$. Finally, 
note that $g'=xg_ax^{-1}=xs_au_a x^{-1}=s_a(xu_ax^{-1})$ since $x\in 
C_\bG^\circ(s_a)$. Hence, $s_a$ also is the semisimple part of $g'$.
\end{proof}

\begin{rem} \label{scad1a} Let $\Sigma$ be an $F$-stable conjugacy class
of $\bG$. Let $\Sigma_{p'}$ be the set of semisimple parts of elements of
$\Sigma$; then $\Sigma_{p'}$ certainly is an $F$-stable conjugacy class of
$\bG$. Note that, in general, not every element in $\Sigma_{p'}^F$ is the 
semisimple part of an element in~$\Sigma^F$. But, at least, in the 
following special situation (which will be important for us later on;
see \ref{lem53} below) this will be the case. Let $g\in\Sigma^F$ and 
write $g=su=us$ where $s\in \Sigma_{p'}^F$ and $u\in \bG^F$ is unipotent.
Assume that $u$ is regular unipotent in $C_\bG^\circ(s)$. (See 
\cite[\S 12.2]{rDiMi} for general properties of regular unipotent 
elements.) We claim that, in this case, every element in $\Sigma_{p'}^F$ 
arises as the semisimple part of an element in~$\Sigma^F$. 

To see this, let $s'\in \Sigma_{p'}^F$ be arbitrary. It is known that regular 
unipotent elements always exist in connected reductive groups, and that 
they form a single conjugacy class. Thus, we can find a regular unipotent 
element $u'\in C_\bG^\circ(s')$ such that $F(u')=u'$. Since $s,s'\in 
\Sigma_{p'}$, there exists some $x\in \bG$ such that $s=xs'x^{-1}$; then 
the element $xu'x^{-1}$ is regular unipotent in $C_\bG^\circ(s)$. Hence, 
$u$ and $xu'x^{-1}$ are conjugate in $C_\bG^\circ(s)$, and so the elements 
$su$ and $s'u'$ are conjugate in $\bG$. Thus, $s'$ is the semisimple
part of $g':=s'u'\in \Sigma^F$.
\end{rem}

\section{On the subspaces $\mbox{CF}(\bG^F\mid \Sigma)$} \label{subsp}

We keep the basic assumptions of the previous section. Let us now fix an 
$F$-stable conjugacy class $\Sigma$ of $\bG$. We denote by $\CF(\bG^F \mid
\Sigma) \subseteq \CF(\bG^F)$ the subspace consisting of all class functions 
that take value~$0$ on elements outside of~$\Sigma^F$. Then the general 
aim will be to find a basis $\{f_i\}$ of $\CF(\bG^F\mid\Sigma)$ such that
\begin{itemize}
\item[$(\Sigma)_1$] the values $f_i(g)$ for $g\in \Sigma^F$ can be 
computed explicitly and
\item[$(\Sigma)_2$] the decomposition of each $f_i$ as a linear combination 
of $\Irr(\bG^F)$ is known.
\end{itemize}
Assume that we have found such a basis $\{f_i\}$. Let $g\in \Sigma^F$ and 
$C$ be the $\bG^F$-conjugacy class of~$g$. Let $\varepsilon_C\in\CF(\bG^F)$ 
be the indicator function of $C$. Using $(\Sigma)_1$, we can express 
$\varepsilon_C$ as a linear combination of the functions $f_i$. Now let 
$\rho\in\Irr(\bG^F)$. By $(\Sigma)_2$, we know the inner products $\langle 
\rho,f_i\rangle_{\bG^F}=\overline{\langle f_i,\rho\rangle}_{\bG^F}$; 
hence, $\rho(g)=|C_\bG(g)^F| \langle \rho,\varepsilon_C\rangle_{\bG^F}$ 
can be computed.

As already mentioned in the introduction, Lusztig \cite{L7} presents a 
general strategy for finding a basis $\{f_i\}$ as above, but this involves 
the delicate issue of fixing normalisations of characteristic functions of 
$F$-invariant character sheaves on $\bG$. In this section, we discuss
some special cases where this issue takes a much simpler form.

First of all, a natural candidate for a basis element of $\CF(\bG^F\mid 
\Sigma)$ is the indicator function $\varepsilon_{\Sigma^F}$ of the whole
set $\Sigma^F$. If $s\in \Sigma_{p'}^F$ (notation as in Remark~\ref{scad1a})
and $C_\bG(s)$ is not connected, then we obtain a slight refinement of 
$\varepsilon_{\Sigma^F}$ by taking the indicator function on the set of 
those elements $g\in \Sigma^F$ such that the semisimple part of $g$ is 
$\bG^F$-conjugate to~$s$. Again, it turns out that this refinement is a 
uniform function; this is the subject of the following section. Now we shall 
consider further candidates for functions in $\mbox{CF}(\bG^F\mid \Sigma)$.

\begin{abs} \label{cuspcs1} We recall some facts about (cuspidal) character 
sheaves; see Lusztig \cite{L7} and further references there. (See
also \cite[2.7.24--2.7.27]{gema} for an informal introduction.) Let $A$ 
be a character sheaf on~$\bG$ such that $F^*A \cong A$. Then the choice 
of an isomorphism $F^*A\cong A$ gives rise to a non-zero class function 
$\chi_A\in \CF(\bG^F)$ which is called a ``characteristic function'' 
of $A$; it is well-defined up to multiplication by a scalar. Assume now that
\begin{equation*}
\mbox{$\bG$ is semisimple and $A$ is a cuspidal character sheaf}.\tag{a}
\end{equation*}
Then there is a unique $F$-stable conjugacy class $\Sigma$ of $\bG$ such 
that $\chi_A(g)=0$ for all $g\in \bG^F \setminus \Sigma^F$, that is, $\chi_A
\in \CF(\bG^F\mid \Sigma)$. (This follows from the fact that $A$ is 
``clean'' \cite{L10}.) Let us fix a representative $g_0\in \Sigma^F$. Then
$A$ corresponds to an irreducible character $\psi \in \Irr(A_\bG(g_0))$ 
which is invariant under the action of $F$ on $A_\bG(g_0)$. Let us assume 
that $\psi$ is a linear character. (This assumption will be satisfied in 
all examples that we consider; it implies that $\psi$ is constant on the 
$F$-conjugacy classes of $A_\bG(g_0)$.) Then the isomorphism $F^*A\cong A$ 
can be chosen such that the values of the corresponding characteristic 
function $\chi_A$ on $\Sigma^F$ are given by 
\begin{equation*}
\chi_A(g)=q^{(\dim \bG-\dim\Sigma)/2} \psi(a) \qquad\mbox{if $g\in C_a
\subseteq \Sigma^F$ where $a\in A_\bG(g_0)$}. \tag{b}
\end{equation*}
(See \cite[4.2]{pamq} and the further references there.) We also set 
$\lambda_A:=\psi(\bar{g}_0)$ where $\bar{g}_0$ denotes the image of $g_0$ 
in $A_\bG(g_0)$. This number has the following alternative interpretation 
in terms of the twisting operator $t_1$ defined in \ref{shintani}: 
\begin{equation*}
t_1^*(\chi_A):=\chi_A\circ t_1=\lambda_A \chi_A; \qquad
\mbox{see Shoji \cite[3.3, 3.8]{S2}}. \tag{c}
\end{equation*}
Hence, $\chi_A\in \CF(\bG^F)$ is an eigenvector for the operator 
$t_1^*\colon \CF(\bG^F)\rightarrow \CF(\bG^F)$ induced by 
$t_1$, where the corresponding eigenvalue is given by~$\lambda_A$. 
\end{abs}

Thus, if $\Sigma$ is the supporting set of an $F$-invariant cuspidal
character sheaf $A$ on $\bG$, then~$\chi_A$ is a natural candidate for 
a basis element of $\CF(\bG^F\mid \Sigma)$, where the remaining
problem is to find the decomposition of $\chi_A$ as a linear combination
of $\Irr(\bG^F)$. (That problem is solved in many cases, but not in
complete generality; see \cite{L7}, \cite{pamq}). 

\begin{exmp} \label{expd4} Let $p=2$ and $\bG$ be simple of type $D_4$. 
By \cite[Prop.~19.3(d'')]{L2d} (and its proof), there is unique cuspidal
character sheaf $A_0$ on $\bG$; it is automatically $F$-invariant. 
Furthermore, the support of $A_0$ is given by $\Sigma=\cO_0$ where $\cO_0$ 
is the conjugacy class of regular unipotent elements. Let us fix a 
representative $u_0\in \cO_0^F$. By \cite[Table~8.5a (p.~126)]{LiSe}, we 
have $\dim C_\bG(u_0)=4$ and $A_\bG(u_0)\cong \Z/2\Z$. Since~$A_0$ 
corresponds to the non-trivial character of $A_\bG(u_0)$, the function 
$\chi_{A_0}\in \mbox{CF}_{\text{uni}}(\bG^F \mid \cO_0)$ (see \ref{cuspcs1}) 
is given by 
\[\chi_{A_0}(u)=\left\{\begin{array}{rl} q^2 & \quad \mbox{if $u\in
\cO_0^F$ is $\bG^F$-conjugate to $u_0$},\\ -q^2 & \quad \mbox{if $u\in 
\cO_0^F$ is not $\bG^F$-conjugate to $u_0$}.
\end{array}\right.\]
Thus, $\chi_{A_0}$ and the indicator function $\varepsilon_{\cO_0^F}$ form 
a basis of $\mbox{CF}(\bG^F\mid\cO_0)$. It remains to address the
following two issues (see Proposition~\ref{d4split} below):
\begin{itemize}
\item[1)] find the decomposition of $\chi_{A_0}$ into irreducible 
characters of $\bG^F$, 
\item[2)] specify the choice of a representative $u_0\in \cO_0^F$. 
\end{itemize}
\end{exmp}

\begin{exmp} \label{expe6} Let $p\neq 3$ and $\bG$ be simple, adjoint of
type $E_6$. By \cite[Cor.~20.4]{L2d} (and its proof) there are $6$
cuspidal character sheaves on $\bG$. They all have the same support;
it is given by an $F$-stable conjugacy class $\Sigma$ where $A_\bG(g_0)
\cong \Z/3\Z\times \Z/3\Z$ for $g_0\in \Sigma^F$. If~$F$ acts trivially
on $A_\bG(g_0)$, then $\dim \CF(\bG^F\mid \Sigma)=9$. The 
characteristic functions of the $6$ cuspidal character sheaves yield $6$
basis functions in $\CF(\bG^F\mid \Sigma)$. So we still need to 
find three further functions; these will be provided by the results in 
the following section.
\end{exmp}


For $\bG$ of type $E_6$, we will also need some information about the 
subspaces $\mbox{CF}(\bL^F\mid \Sigma)$ where $\bL\subseteq \bG$ is a 
regular subgroup (as in~\ref{exp01}) and $\Sigma$ is a unipotent class 
of~$\bL$. Below we shall consider some examples which are relevant for
the discussion in Section~\ref{finale6}. First, some preparations.

\begin{abs} \label{green0} Let $\bG_{\text{uni}}$ be the set of unipotent 
elements of $\bG$ and $\mbox{CF}_{\text{uni}}(\bG^F)$ be the vector space 
of all functions $\bG_{\text{uni}}^F\rightarrow\C$ that are constant on
the $\bG^F$-conjugacy classes in $\bG_{\text{uni}}^F$. By restriction, 
we obtain a linear map $\pi_{\text{uni}}^\bG \colon \mbox{CF}(\bG^F)
\rightarrow \mbox{CF}_{\text{uni}}(\bG^F)$. It will be convenient to
regard the functions in $\mbox{CF}_{\text{uni}}(\bG^F)$ as class functions 
on all of $\bG^F$, by letting their value be $0$ on elements outside 
$\bG_{\text{uni}}^F$. With this convention, we have 
\[\pi_{\text{uni}}^\bG(f)=\gamma_{\text{uni}}^\bG{\cdot}f\qquad
\mbox{for any $f \in \mbox{CF}(\bG^F)$},\]
where the function $\gamma_{\text{uni}}^\bG\in \mbox{CF}(\bG^F)$ has been
defined in Example~\ref{pconst}. 
\end{abs}

The following two results are probably known to the experts. Since
we have not found a convenient reference, we sketch the proofs.

\begin{lem} \label{clp2} Assume that $p=2$, that the center of $\bG$ is
connected and that $\bG/Z(\bG)$ only has simple components of type $A_n$, 
$B_n$, $C_n$ or $D_n$. Then the restrictions of the unipotent characters 
of $\bG^F$ to $\bG_{\operatorname{uni}}^F$ span
$\operatorname{CF}_{\operatorname{uni}}(\bG^F)$.
\end{lem}

\begin{proof} Clearly, $\mbox{CF}_{\text{uni}}(\bG^F)$ is spanned by the 
functions $\pi_{\text{uni}}^\bG(\rho)$ for $\rho \in \Irr(\bG^F)$. We must 
show that it is sufficient to take $\pi_{\text{uni}}^\bG(\rho)$ for 
$\rho \in \Irr(\bG^F)$ unipotent. The main point of the argument is that 
our assumptions on $\bG$ imply that $C_{\bG^*}(s)$ is a regular subgroup
for every semisimple element $s\in {\bG^*}^{F^*}$, where $\bG^*$ is a
group dual to $\bG$ with dual Frobenius map $F^*\colon \bG^*\rightarrow 
\bG^*$. (This follows from \cite[\S 2.2, \S 2.3]{bs1}, 
and this has already been used in the proof of \cite[8.7]{Lu1}.) 
Hence, we are in the ``Levi case'' of Lusztig's Jordan decomposition of 
characters; see \cite[\S 11.4]{rDiMi}, \cite[\S 3.3]{gema}. This means
that every $\rho \in \Irr(\bG^F)$ is of the form $\rho=\pm R_{\bL\subseteq 
\bP}^\bG(\lambda{\cdot}\psi)$ where $\bL$ is an $F$-stable Levi complement 
of a parabolic subgroup $\bP \subseteq \bG$, $\lambda\colon \bL^F
\rightarrow \C^\times$ is a $p$-constant linear character, and $\psi 
\in \Irr(\bL^F)$ is a unipotent character. Now recall from \ref{green0}
that $\pi_{\text{uni}}^\bG(\rho)=\gamma_{\text{uni}}^\bG{\cdot}\rho$. 
Since $\gamma_{\text{uni}}^\bG$ is $p$-constant, we have 
\[ \pi_{\text{uni}}^\bG(\rho)=\gamma_{\text{uni}}^\bG{\cdot}\rho= \pm 
\gamma_{\text{uni}}^\bG{\cdot} R_{\bL\subseteq \bP}^\bG\bigl(\lambda{\cdot} 
\psi\bigr)=\pm R_{\bL\subseteq \bP}^\bG\bigl(\gamma_{\text{uni}}^\bL{\cdot} 
(\lambda{\cdot} \psi)\bigr);\]
see \cite[Prop.~3.3.16]{gema} for the last equality. Since $\lambda$ is 
also $p$-constant, we certainly have $\gamma_{\text{uni}}^\bL{\cdot}
(\lambda{\cdot} \psi)=\gamma_{\text{uni}}^\bL {\cdot} \psi$. Hence, 
using once more \cite[Prop.~3.3.16]{gema}, we conclude that 
\[ \pi_{\text{uni}}^\bG(\rho)=\pm R_{\bL\subseteq \bP}^\bG\bigl(
\gamma_{\text{uni}}^\bL{\cdot} \psi \bigr)=\pm \gamma_{\text{uni}}^\bG
{\cdot} R_{\bL\subseteq \bP}^\bG(\psi)=\pm \pi_{\text{uni}}^\bG\bigl(
R_{\bL\subseteq \bP}^\bG(\psi)\bigr).\]
This completes the proof, since it is known that $R_{\bL\subseteq 
\bP}^\bG(\psi)$ is a linear combination of unipotent characters of
$\bG^F$; see \cite[Prop.~3.3.20]{gema}.
\end{proof}


\begin{prop} \label{clp2b} Let $p=2$ and $\bG$ be simple of type
$A_n$, $B_n$, $C_n$ or $D_n$. Then the restrictions of the unipotent 
characters of $\bG^F$ to $\bG_{\operatorname{uni}}^F$ form a basis of 
$\operatorname{CF}_{\operatorname{uni}}(\bG^F)$.
\end{prop}

\begin{proof} Clearly, $\dim \operatorname{CF}_{\operatorname{uni}}(\bG^F)$
is equal to the number of $\bG^F$-conjugacy classes of unipotent elements
of $\bG^F$. By \cite[7.14, 8.7]{Lu1}, the latter number is also equal to 
the number of unipotent characters of $\bG^F$. So the assertion follows
from Lemma~\ref{clp2}.
\end{proof}

\begin{abs} \label{lfrom} Let $\cO$ be an $F$-stable unipotent class
of $\bG$ and assume that $\cO^F$ splits into conjugacy classes $C_1,
\ldots C_r$ of $\bG^F$. The following discussion will be helpful in 
distinguishing these $\bG^F$-classes. Let $\bB\subseteq \bG$ be an 
$F$-stable Borel subgroup and $\bW$ be the Weyl group of $\bG$ with respect 
to an $F$-stable maximal torus contained in $\bB$. Let $w\in \bW^F$ and 
consider the Bruhat cell $\bB w\bB$. The cardinality $|\cO^F\cap \bB^Fw
\bB^F|$ can be explicitly determined in terms of a formula involving 
the Green functions of $\bG^F$, Lusztig's non-abelian Fourier matrices and 
the character values of the Hecke algebra of $\bW^F$; see \cite[1.2]{Lfrom} 
for details. (In {\sf CHEVIE}, this can be done as explained in 
\cite[\S 6]{gap3jm}.) Assume now that $\cO \cap \bB w\bB\neq \varnothing$.
Then~---at least in certain favorable situations (see below for 
examples)~---~the notation can be chosen such that 
\begin{equation*}
C_1\cap \bB^Fw\bB^F\neq \varnothing \qquad\mbox{and} \qquad 
C_i\cap \bB^Fw\bB^F=\varnothing \quad\mbox{for $i=2,\ldots,r$}.\tag{a}
\end{equation*}
Consequently, we have $|C_1\cap \bB^F w\bB^F|=|\cO^F\cap \bB^F w \bB^F|$,
and the right hand is known; furthermore, $|C_i\cap \bB^F w\bB^F|=0$ for
$i=2,\ldots,r$. On the other hand, $|C_i\cap \bB^F w \bB^F|$ can also be 
expressed in terms of a formula involving the values of certain unipotent 
characters of $\bG^F$ and, again, the character values of the Hecke algebra 
of~$\bW^F$. If not yet all values of the unipotent characters of $\bG^F$ on 
$C_1,\ldots,C_r$ are known, then that formula yields $r$ linear relations 
for the unknown values. (This argument has already been used, for example, 
in the proof of \cite[Prop.~6.5]{pamq}; see also Hetz \cite{Het3} and 
further references there.)
\end{abs}

For groups of small rank ($\leq 7$, say), the Green functions $Q_\bT^\bG$ 
in Example~\ref{pconst} typically span a significant subspace of 
$\mbox{CF}_{\text{uni}}(\bG^F)$, and the values of $Q_\bT^\bG$ can be 
explicitly computed using the methods in \cite{ekay} (especially for the
case where~$q$ is a power of a small prime~$p$). In the examples below, we
will assume that tables with the values of the functions $Q_\bT^\bG$ are
known (but we will not print these tables). 

We shall now consider groups $\bG$ of type $D_n$, where we use the following 
labelling of simple roots in a root system of $\bG$:
\begin{center}
\begin{picture}(280,42)
\put( 15,22){$D_n$}
\put( 15,8){\scriptsize ($n{\geq} 4$)}
\put(132, 3){$\alpha_2$}
\put( 91,33){$\alpha_1$}
\put(121,33){$\alpha_3$}
\put(151,33){$\alpha_4$}
\put(211,33){$\alpha_n$}
\put(125, 5){\circle*{5}}
\put( 95,25){\circle*{5}}
\put(125,25){\circle*{5}}
\put(155,25){\circle*{5}}
\put(215,25){\circle*{5}}
\put(125,25){\line(0,-1){20}}
\put( 95,25){\line(1,0){30}}
\put(125,25){\line(1,0){30}}
\put(155,25){\line(1,0){12}}
\put(177,25){\circle*{2}}
\put(185,25){\circle*{2}}
\put(193,25){\circle*{2}}
\put(203,25){\line(1,0){12}}
\end{picture}
\end{center}
In this case, the unipotent characters of $\bG^F$ are 
parametrized by (equivalence classes of) certain symbols of rank~$n$; 
see Lusztig \cite[Theorem~8.2]{Lu1}. Recall from \cite[\S 3]{Lu1} that
symbols are unordered pairs $\Lambda=(S,T)$ of finite subsets of
$\Z_{\geq 0}$. The rank of $\Lambda$ is defined as in \cite[2.6.4]{Lu1}, 
the defect of $\Lambda$ is the absolute value of $|S|-|T|$. We denote the 
unipotent character corresponding to $\Lambda=(S,T)$ by $\rho_{(S,T)}$.



\begin{abs} \label{d4intro} Let $p=2$ and $\bG$ be simple of type $D_4$.

(a) Assume that $\bG^F=D_4(q)$ (split type). By \cite[Table~8.5a 
(p.~126)]{LiSe}, there are $12$ unipotent classes in~$\bG$, which split 
into~$14$ unipotent classes in~$\bG^F=D_4(q)$. Hence, we also have 
$\dim\mbox{CF}_{\text{uni}}(\bG^F)=14$. The unipotent characters of
$\bG^F$ are parametrized by symbols of rank $4$ and defect $d\equiv 0 
\bmod 4$, where each symbol of the form $(S,S)$ is repeated twice. We
define a class function $f_0\in \mbox{CF}(\bG^F)$ by 
\[ \textstyle f_0:=\frac{1}{2}(\rho_{(13,02)}-\rho_{(23,01)}-\rho_{(12,03)}
+\rho_{(0123,-)})\]
(where $(013,2)$ stands for the symbol $\Lambda=(\{0,1,3\},\{2\})$ etc.)

(b) Assume that $\bG^F={^2\!D}_4(q)$ (twisted type). Referring again to 
\cite[Table~8.5a (p.~126)]{LiSe}, we see that only~$8$ out of the $12$ 
unipotent classes of $\bG$ are $F$-stable; these split into $10$ unipotent 
classes in $\bG^F={^2\!D}_4(q)$. Hence, we also have $\dim
\mbox{CF}_{\text{uni}}(\bG^F)=10$. Now the unipotent characters of $\bG^F$ 
are parametrized by symbols of rank $4$ and defect $d\equiv 2 \bmod 4$. We
define a class function $f_0\in \mbox{CF}(\bG^F)$ by 
\[ \textstyle f_0:=\frac{1}{2}(\rho_{(123,0)} -\rho_{(012,3)}
+\rho_{(013,2)} -\rho_{(023,1)}.\]
The significance of these definitions is as follows. By Main Theorem~4.23 
of \cite{LuB}, the decomposition of each generalised character $R_{\bT,
1}^\bG$ as a linear combination of unipotent characters is known. (In
{\sf CHEVIE} \cite{gap3jm}, this information is readily available via the 
functions {\tt UnipotentCharacters} and {\tt DeligneLusztigCharacter}; 
see also \cite[\S 2.4]{gema}.) For $f_0$ as defined above, we find that 
$\langle f_0,R_{\bT,1}^\bG \rangle_{\bG^F}=0$ for all~$\bT$; furthermore, 
every unipotent character of~$\bG^F$ is uniquely a linear combination of 
$f_0$ and the various $R_{\bT,1}^\bG$. Hence, in order to determine the 
values of the unipotent characters on unipotent elements (and even on 
all elements of $\bG^F$), it remains to determine the values of~$f_0$.
\end{abs}

In the split case, the following result is a very special case
of Shoji \cite[Theorem~6.2]{S4}.

\begin{prop} \label{d4split} In the setting of \ref{d4intro}, let $\cO_0$ 
be the conjugacy class of regular unipotent elements of~$\bG$ 
and $u_0:=x_{\alpha_1}(1)x_{\alpha_2}(1)x_{\alpha_3}(1)x_{\alpha_4}(1)\in 
\cO_0^F$. Then, for any $g \in \bG^F$, we have 
\[ f_0(g)=\left\{\begin{array}{cl} q^2 & \quad\mbox{if $g \in \cO_0^F$
and $g$ is $\bG^F$-conjugate to $u_0$},\\
-q^2 & \quad\mbox{if $g \in \cO_0^F$ and $g$ is not $\bG^F$-conjugate to
$u_0$},\\ 0 & \quad \mbox{otherwise}.\end{array}\right.\]
Thus, we have $f_0=\chi_{A_0}$ where $\chi_{A_0}$ is the 
characteristic function in Example~\ref{expd4}.
\end{prop}

\begin{proof} Assume first that $\bG^F=D_4(q)$. As already discussed in 
\cite[Example~4.4]{padua} (or in \cite{S4}), we do have $f_0=\chi_{A_0}$ 
in this case. Now let $\bG^F={^2\!D}_4(q)$. Then we need a different 
argument since, in \cite{padua} and in \cite{S4}, only groups of split 
type are considered. (The following argument would also work, almost 
verbatim, for the split case.) As already mentioned, $\cO_0^F$ splits 
into two classes in~$\bG^F$, with centraliser orders $2q^4,2q^4$; let 
$u_0,u_1\in \cO_0^F$ (with $u_0$ as above) be representatives of these
two classes. We define $f_0'\in \mbox{CF}_{\text{uni}}(\bG^F)$ by
\[f_0'(u_0):=q^2, \qquad f_0'(u_1):=-q^2 \qquad \mbox{and} \qquad f_0'(u)
:=0 \quad \mbox{for $u \in \bG_{\text{uni}}^F\setminus \cO_0^F$}.\]
(Note that $f_0'=\chi_{A_0}$, but the following argument works without
reference at all to the theory of character sheaves.) As usual, we regard 
$f_0'$ as a function on all of $\bG^F$ by letting its value be~$0$ on 
non-unipotent elements; then $\langle f_0',f_0' \rangle_{\bG^F}=1$. We 
must show that $f_0=f_0'$. 

Now, the distinct Green functions $Q_\bT^\bG$ are parametrized by the
$F$-conjugacy classes of the Weyl group~$\bW$. (See, e.g., 
\cite[\S 2.3]{gema}.) One checks that there are $9$ such $F$-conjugacy 
classes. Hence, since distinct Green functions are linearly independent and 
since $\dim \mbox{CF}_{\text{uni}}(\bG^F)=10$, the Green functions span a 
subspace of co-dimension~$1$ in $\mbox{CF}_{\text{uni}}(\bG^F)$. Using the 
knowledge of the values of the Green functions, we find that $\langle f_0',
Q_\bT^\bG \rangle_{\bG^F}=0$ for all~$\bT$ (where, as usual, we regard each 
$Q_\bT^\bG$ as a function on all of~$\bG^F$ with value~$0$ on non-unipotent 
elements). Thus, we have:
\begin{equation*}
\mbox{The Green functions $Q_\bT^\bG$ together with $f_0'$ form a basis of 
$\mbox{CF}_{\text{uni}}(\bG^F)$}. \tag{a}
\end{equation*}
Next recall from \ref{d4intro} that $\langle f_0,R_{\bT,1}^\bG\rangle_{\bG^F}
=0$ for all~$\bT$. Since $f_0$ is a linear combination of unipotent 
characters, it follows that $f_0$ is orthogonal to all uniform functions 
on $\bG^F$. So, by Example~\ref{pconst}, we have $\langle f_0,Q_\bT^\bG \rangle_{\bG^F}=0$ for all $\bT$. On the 
other hand, by Proposition~\ref{clp2b}, the restriction of $f_0$ to 
$\bG_{\text{uni}}^F$ must be non-zero. Hence, using (a) and the fact 
that $\langle f_0',Q_{\bT}^\bG\rangle_{\bG^F}=0$ for all $\bT$, we deduce
that $\pi_{\text{uni}}^\bG(f_0)=cf_0'$ for some $0\neq c\in \C$. In order
to determine the scalar~$c$, we use the discussion in \ref{lfrom}. 
For $i=0,1$ let $C_i$ be the $\bG^F$-conjugacy class of~$u_i$; 
thus, $\cO_0^F=C_0\cup C_1$. Let $w_c\in \bW$ be a Coxeter element, such
that $F(w_c)=w_c$. As discussed in \cite[Example~4.9]{pamq}, we have 
\begin{equation*}
|\cO_0^F \cap \bB^F w_c \bB^F|=|\bB^F|\qquad \mbox{and} \qquad 
C_i\cap \bB^F w_c \bB^F= \varnothing \quad\mbox{for $i=0$ or $i=1$}.\tag{b}
\end{equation*}
By the argument of Hetz \cite[Lemma~7.2]{Het4} (which is based on Lusztig
\cite[2.4]{Lfrom}), one sees that a certain $\bG^F$-conjugate of $u_0$ 
is contained in $\bB^F w_c \bB^F$. Hence, $C_0\cap \bB^F w_c \bB^F\neq 
\varnothing$. Using~(b), we deduce that $|C_0\cap \bB^Fw_c\bB^F|=|\bB^F|$. 
As explained in \ref{lfrom}, this yields a linear relation involving the 
scalar~$c$; by explicitly working out the terms in that relation (we omit 
the details), we find that $c=1$. Thus, $\pi_{\text{uni}}^\bG(f_0)=f_0'$. 
Finally, since $\langle f_0,f_0\rangle_{\bG^F}=\langle f_0',f_0'
\rangle_{\bG^F}=1$, we can also conclude that $f_0(g)=0$ for all
non-unipotent elements $g\in \bG^F$; thus, $f_0=f_0'$ as desired.
\end{proof}

We shall also need an analogous result for groups of type $D_5$. In order 
to avoid too much interruption of the current discussion, this will 
be done in an Appendix (see Section~\ref{appdx}).

\section{On non-simply-connected semisimple groups} \label{scad}

Assume that $\bG$ is semisimple. If $\bG$ is not of simply connected type, 
then it may happen that the commutator subgroup of $\bG^F$ is strictly 
smaller than $\bG^F$; furthermore, the centraliser of a semisimple element
of $\bG$ may not be connected. In this section we provide some tools for 
dealing with this situation; this is mostly based on Steinberg \cite{St68}. 
For this purpose, it will be convenient to also fix a connected reductive 
algebraic group~$\tbG$ over $k$ and a surjective homomorphism $\pi\colon 
\tbG \rightarrow \bG$, with the following properties:
\begin{itemize}
\item There is a Frobenius map $\tF\colon \tbG \rightarrow \tbG$ 
such that $F\circ \pi= \pi\circ \tF$;
\item the kernel $\ker(\pi)$ is connected and contained in the center
$Z(\tbG)$ of $\tbG$;
\item the derived subgroup $\tbG_{\text{der}}\subseteq \tbG$ is 
semisimple of simply connected type.  
\end{itemize}
(The fact that this always exists is due to Asai; see 
\cite[Prop.~1.7.13]{gema}.) We begin by collecting some useful properties 
about $\pi\colon \tbG\rightarrow \bG$.

\begin{rem} \label{rem10} The restriction of $\pi$ to $\tbG_{\text{der}}$ 
will be denoted by $\pi_{\text{der}}\colon \tbG_{\text{der}} \rightarrow 
\bG$. Since $\bG$ is semisimple, $\pi_{\text{der}}$ is surjective and 
$\bK:=\ker(\pi_{\text{der}}) \subseteq Z(\tbG_{\text{der}})$ is finite. 
(For example, if $\bG$ is simple and not of type $A_n$, then $\bK$ has
order $1,2,3$ or $4$.) Since $\ker(\pi)$ is connected, we have 
$\pi(\tbG^\tF)=\bG^F$. But note that we do not necessarily have 
$\pi(\tbG_{\text{der}}^\tF)=\bG^F$. Firstly, by Steinberg 
\cite[Cor.~12.6]{St68}, we have 
\begin{equation*}
\pi(\tbG_{\text{der}}^\tF)=\bG^F_u:=\mbox{subgroup of $\bG^F$ generated by 
the unipotent elements of $\bG^F$}.\tag{a}
\end{equation*}
Secondly, $\bG^F_u$ is a normal subgroup of $\bG^F$ and there is a 
canonical isomorphism 
\begin{equation*}
\bG^F/\bG^F_u=\bG^F/\pi(\tbG_{\text{der}}^\tF)
\stackrel{\sim}{\longrightarrow} \bK/\{z^{-1}\tF(z) \mid z\in \bK\},\tag{b}
\end{equation*}
defined as follows (see \cite[Prop.~1.4.13]{gema}): Let $g \in \bG^F$ and 
$\dot{g}\in \tbG_{\text{der}}$ be such that $\pi(\dot{g})=g$; then $u:=
\dot{g}^{-1}\tF(\dot{g})\in \bK$ and the coset of $g$ is mapped to the 
coset of~$u$.
\end{rem}

\begin{rem} \label{rem10a} Let $s\in \bG$ be semisimple and $\tis\in 
\tbG$ be such that $\pi(\tis)=s$; note that~$\tis$ is also semisimple. 
Since $\tbG_{\text{der}}$ is simply connected, it is known that the 
centraliser $C_{\tbG}(\tis)$ is connected; see Steinberg 
\cite[\S 8]{St68}. However, $C_\bG(s)$ need not be connected, but 
the group of components $A_\bG(s)=C_\bG(s)/C_\bG^\circ(s)$ is always 
isomorphic to a subgroup of $\bK=\ker(\pi_{\text{der}})$. Slightly more 
generally, we have:
\end{rem}

\begin{lem} \label{scad3b} Let $\tilde{g}\in \tbG^\tF$ be such that $C_\tbG
(\tilde{g})$ is connected, and set $g:=\pi(\tilde{g}) \in \bG$. Then 
$A_\bG(g)$ is isomorphic to a subgroup of $\bK$, and $A_\bG(g)^F$ is 
isomorphic to a subgroup of $\bK^\tF$.
\end{lem}

\begin{proof} This is contained in \cite[9.1]{St68}; we recall the main 
ingredients. First note that $\pi(C_\tbG(\tilde{g}))=\pi(C_{\tbG}^\circ
(\tilde{g}))=C_\bG^\circ(g)$; see \cite[1.3.10(e)]{gema}. Then there is 
a canonical isomorphism 
\[ \delta\colon A_\bG(g) \stackrel{\sim}{\longrightarrow} \{h^{-1}
\tilde{g}^{-1}h\tilde{g} \mid h \in \tbG\}\cap \ker(\pi)\]
defined as follows. Let $x\in C_\bG(g)$ and $\dot{x}\in \tbG$ be such 
that $\pi(\dot{x})=x$. Then $z_{\dot{x}}:=\dot{x}^{-1}\tilde{g}^{-1}
\dot{x}\tilde{g} \in \ker(\pi)$ and $\delta$ sends the image of $x=\pi(
\dot{x})$ in $A_\bG(g)$ to~$z_{\dot{x}}$; note that $z_{\dot{x}}\in 
\tbG_{\text{der}}\cap \ker(\pi)=\bK$. (The construction of 
$\delta$ is based on the purely group-theoretical result in 
\cite[4.5]{St68}; see also \cite[Lemma~1.1.9]{gema}). Now one 
checks that $\delta\circ F=\tF\circ \delta$. Hence, $A_\bG(g)^F$ is 
isomorphic to a subgroup of~$\bK^\tF$.
\end{proof}

\begin{rem} \label{scad3c} Let $\tilde{g}\in \tbG$ and $g:=\pi(\tilde{g}) 
\in \bG$. If $C_{\tbG}(\tilde{g})$ is not connected, then $\pi(C_{\tbG}
(\tilde{g}))\supseteq \pi(C_{\tbG}^\circ(\tilde{g}))=C_\bG^\circ(g)$ and 
$C_\bG(g)/\pi(C_{\tbG}(\tilde{g}))$ is a factor group of $A_\bG(g)$. By 
exactly the same argument as above (using \cite[4.5]{St68}), we still 
obtain an injective homomorphism
\[ \delta\colon C_\bG(g)/\pi(C_{\tbG}(\tilde{g})) \hookrightarrow \bK,\]
Since $\ker(\pi)$ is connected, we have $\ker(\pi) \subseteq C_\tbG^\circ
(\tilde{g})$ and so $|A_\tbG(\tilde{g})|=|\pi(C_\tbG(\tilde{g})):
C_\bG^\circ(g)|$.
\end{rem}

%

The following technical result will be useful further on.

\begin{lem} \label{lem00} Assume  that $\tF$ acts trivially on $\bK$. 
Let $g_1,g_2\in \bG^F$ be elements that are conjugate in $\bG$. 
If $g_1^{-1}g_2\in \pi(\tbG_{\operatorname{der}}^\tF)$ and $C_\tbG
(\tilde{g}_1)$ is connected (where $\tilde{g}_1\in \tbG$ is such that 
$\pi(\tilde{g}_1)=g_1)$, then $g_1,g_2$ are already conjugate in $\bG^F$. 
\end{lem}

\begin{proof} Let $\bC$ be the $\bG$-conjugacy class containing the
elements $g_1,g_2$. Since $\pi(\tbG_{\text{der}})=\bG$, there exists a 
$\tbG_{\text{der}}$-conjugacy class $\tilde{\bC}$ such that 
$\pi(\tilde{\bC})=\bC$. For $i=1,2$ let $x_i\in \tilde{\bC}$ be such 
that $\pi(x_i)=g_i$; then $u_i:=x_i^{-1} \tF(x_i) \in \bK$. Note that 
the elements $x_1,x_2$ are not necessarily fixed by~$\tF$. We will slightly
modify them so as to obtain elements in $\tbG^\tF$. This is done as follows. 
Since $\tF$ acts trivially on $\bK$, we have a canonical isomorphism 
\[ \bG^F/\pi(\tbG_{\text{der}}^\tF)\stackrel{\sim}{\longrightarrow} \bK,\]
defined as in Remark~\ref{rem10}(b); note that, under this isomorphism,
the coset of $g_i$ is sent to~$u_i$. But, assuming that $g_1,g_2$ belong 
to the same coset mod $\pi(\tbG_{\text{der}}^\tF)$, we must have $u_1=u_2$. 
Let $y:=x_1^{-1}x_2\in \tbG_{\text{der}}$. Since $\tF(x_i)=x_iu_i$ 
and $u_i\in \bK\subseteq Z(\tbG)$ for $i=1,2$, we obtain $\tF(y)=y$. 
Hence, $x_2=x_1y$ where $y\in \tbG_{\text{der}}^\tF$. 

Since $u_1\in \bK\subseteq \ker(\pi)$ and $\ker(\pi)$ is connected, we 
can further find an element $z\in \ker(\pi)$ such that $u_1=z\tF(z)^{-1}$. 
Since $\ker(\pi)\subseteq Z(\tbG)$, we obtain $\tF(zx_1)=\tF(z)\tF(x_1)=
(u_1^{-1}z)(x_1u_1)=zx_1$ and $\tF(zx_2)=\tF(zx_1y)=\tF(zx_1)y=zx_1y=zx_2$. 
Thus, $zx_1,zx_2\in \tbG^\tF$ where $z\in Z(\tbG)$. Since $x_1,x_2$ are 
conjugate in $\tbG_{\text{der}}$, the elements $zx_1$ and $zx_2$ will 
be conjugate in $\tbG$. 

Finally, since $z\in Z(\tbG)$, $\pi(\tilde{g}_1)=g_1=\pi(x_1)$ and 
$\ker(\pi) \subseteq Z(\tbG)$, we have $C_\tbG(zx_1)=C_\tbG(x_1)=
C_{\tbG}(\tilde{g}_1)$. By assumption, this centraliser is connected.
Consequently, $zx_1\in \tbG^\tF$ and $zx_2\in \tbG^\tF$ are not only 
conjugate in $\tbG$ but already in $\tbG^\tF$ (see~\ref{rem00}). 
But then $g_1=\pi(zx_1)$ and $g_2=\pi(zx_2)$ will be conjugate in 
$\bG^F$, as claimed.
\end{proof}

\begin{exmp} \label{lem00a} Assume that $\tF$ acts trivially on $\bK$.
Let $s\in \bG^F$ be semisimple such that $r:=|A_\bG(s)|=|\bK|$; let $\bC$
be the $\bG$-conjugacy class of~$s$. By Remark~\ref{rem10a} and 
Lemma~\ref{scad3b}, we have $A_\bG(s)\cong \bK$ and $F$ acts trivially 
on $A_\bG(s)$. So $\bC^F$ splits into precisely $r$ conjugacy classes of
$\bG^F$; let $s_1,\ldots,s_r\in \bG^F$ be representatives of these 
classes. Then $s_1,\ldots,s_r$ are also representatives for the cosets
of $\bG^F$ mod $\pi(\tbG_{\operatorname{der}}^\tF)$ and  
\begin{equation*}
C_i:=\bC^F\cap s_i\pi(\tbG_{\operatorname{der}}^\tF) \quad \mbox{is the
$\bG^F$-conjugacy class of $s_i$, for $i=1,\ldots,r$}.\tag{a}
\end{equation*}
Indeed, as in the above proof, we have $r=|\bK|=|\bG^F/
\pi(\tbG_{\text{der}}^\tF)|$. The elements $s_i$ are all conjugate in 
$\bG$, but no two of them are conjugate in $\bG^F$. Hence, they are in 
different cosets mod $\pi(\tbG_{\text{der}}^\tF)$. Thus, (a) holds.
\end{exmp}

We give some further applications of the above results. Let $f\in \CF
(\bG^F)$ be any class function and $s\in \bG^F$ be semisimple. Then we 
define
\[ f^{(s)}:=d^{-1}\sum_\lambda \lambda(s)^{-1} \lambda \cdot f \in
\CF(\bG^F),\]
where the sum runs over all linear characters $\lambda \colon \bG^F
\rightarrow \C^\times$ with $\pi(\tbG_{\text{der}}^\tF)\subseteq 
\ker(\lambda)$, and $d\geq 1$ is the number of such linear characters.

\begin{prop} \label{scad1} Assume that $\tF$ acts trivially on $\bK$. 
Let $\Sigma$ be an $F$-stable conjugacy class of $\bG$ and $f\in \CF
(\bG^F\mid \Sigma)$. Let $s\in \bG^F$ be the semisimple part of an element 
of $\Sigma^F$, and $\Sigma^F_{s}$ be the set of all those $g \in \Sigma^F$ 
such that the semisimple part of $g$ is $\bG^F$-conjugate to~$s$. Note 
that $\Sigma_{s}^F$ is a union of $\bG^F$-conjugacy classes. Then, for 
any $g\in \bG^F$, we have 
\[ f^{(s)}(g)=\left\{\begin{array}{cl} f(g) & \qquad \mbox{if $g \in 
\Sigma_{s}^F$},\\ 0 & \qquad \mbox{otherwise}.\end{array}\right.\]
\end{prop}

\begin{proof} Let $g\in \bG^F$ and write $g=tv=vt$ where $t\in \bG^F$ is 
semisimple and $v \in \bG^F$ is unipotent. By Remark~\ref{rem10}(a),
we have $\lambda(v)=1$ for any $\lambda$ as above, and so $\lambda(g)=
\lambda(t)$. This yields the formula
\[ f^{(s)}(g)=d^{-1}\Bigl(\sum_\lambda \lambda(s^{-1}t)\Bigr)f(g).\]
If $g \in \Sigma_{s}^F$, then $s,t$ will be conjugate in $\bG^F$ and so 
$\lambda(s^{-1}t)=1$. In this case, we obtain $f^{(s)}(g)=f(g)$, as desired.
Now let $g\in \bG^F$ be arbitrary such that $f^{(s)}(g)\neq 0$. First of all,
this implies that $g \in \Sigma^F$ and, consequently, the elements $s,t$ 
are conjugate in $\bG$. By the orthogonality relations for irreducible
characters, the sum $\sum_\lambda \lambda(s^{-1}t)$ will only be non-zero 
if $s^{-1}t \in \pi (\tbG_{\text{der}}^\tF)$. Let $\tis\in \tbG$ be 
semisimple such that $\pi(\tis)=s$; then $C_\tbG(\tis)$ is connected as
already mentioned in Remark~\ref{rem10a}. Hence, Lemma~\ref{lem00} implies 
that $s,t$ are conjugate in $\bG^F$, and so $g \in \Sigma_{s}^F$.
\end{proof}

\begin{cor} \label{scad2} In the setting of Proposition~\ref{scad1},
assume that $f$ is uniform. Then $f^{(s)}$ also is a uniform
function. In particular, the indicator function $\varepsilon_{\Sigma_{s}^F}
\in \operatorname{CF}(\bG^F\mid\Sigma)$ is a uniform function.
\end{cor}


\begin{proof} Let $\lambda \colon \bG^F\rightarrow \C^\times$ be a linear
character with $\pi(\tbG_{\operatorname{der}}^\tF) \subseteq \ker(\lambda)$; 
then $\lambda$ is $p$-constant (see Remark~\ref{rem10}(a)). Hence, $\lambda 
{\cdot} f$ is uniform for any such $\lambda$ (see Example~\ref{pconst}) and, 
consequently, $f^{(s)}$ is uniform. Finally, as already 
mentioned in~\ref{unif}, the indicator function $f:=\varepsilon_{\Sigma^F}
\in \CF(\bG^F)$ is uniform. By Proposition~\ref{scad1}, the indicator 
function $\varepsilon_{\Sigma_s^F} \in \CF(\bG^F)$ is a scalar multiple
of $f^{(s)}$. Hence, $\varepsilon_{\Sigma_s^F}$ is uniform.
\end{proof}

\begin{cor} \label{scad3} Assume that $\tF$ acts trivially on $\bK$. 
Let $C$ be a conjugacy class of $\bG^F$ 
and $g\in C$. If $C_\tbG(\tilde{g})$ is connected (where $\tilde{g}\in 
\tbG$ is such that $\pi(\tilde{g})=g$), then $\varepsilon_C \in \CF(\bG^F)$ 
is a uniform function.
\end{cor}

\begin{proof} Let $\Sigma$ be the $\bG$-conjugacy class of $g$ and $s\in 
\bG^F$ be the semisimple part of $g$; then $C\subseteq \Sigma_s^F\subseteq 
\Sigma^F$. By Corollary~\ref{scad2}, it will be sufficient to show that 
$C=\Sigma_s^F$. So let $g'\in \Sigma_s^F$ be arbitrary. Conjugating $g'$ by 
a suitable element of $\bG^F$ we may assume without loss of generality that 
$s$ also is the semisimple part of~$g'$. Then $g^{-1}g'$ is unipotent and, 
hence, $g^{-1}g'\in \bG^F_u=\pi(\tbG_{\text{der}}^{\tF})$; see 
Remark~\ref{rem10}(a). Since $g,g'\in \Sigma_s^F \subseteq \Sigma$, the 
elements $g,g'$ are conjugate in $\bG$. So $g,g'$ are conjugate in $\bG^F$ 
by Lemma~\ref{lem00}. Thus, $g'\in C$, as desired.
\end{proof}

\section{Groups of adjoint type $E_6$ in characteristic $\neq 3$} 
\label{e6p2}

Throughout this section, let $p\neq 3$ and $\bG$ be simple, adjoint of 
type $E_6$. Let $\bT_0\subseteq \bG$ be a maximal torus and $\bB\subseteq
\bG$ be a Borel subgroup containing~$\bT_0$. Let $\Phi$ be the root system
of $\bG$, where the labelling for the simple roots is chosen as
in the following diagram:
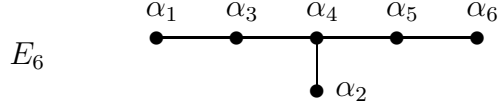
\begin{center}
\begin{picture}(270,42)
\put( 10,15){$E_6$}
\put(132, 3){$\alpha_2$}
\put( 61,33){$\alpha_1$}
\put( 91,33){$\alpha_3$}
\put(121,33){$\alpha_4$}
\put(151,33){$\alpha_5$}
\put(181,33){$\alpha_6$}
\put(125, 5){\circle*{5}}
\put( 65,25){\circle*{5}}
\put( 95,25){\circle*{5}}
\put(125,25){\circle*{5}}
\put(155,25){\circle*{5}}
\put(185,25){\circle*{5}}
\put(125,25){\line(0,-1){20}}
\put( 65,25){\line(1,0){30}}
\put( 95,25){\line(1,0){30}}
\put(125,25){\line(1,0){30}}
\put(155,25){\line(1,0){30}}
\end{picture}
\end{center}
Let $q$ be a power of~$p$ and $F_q\colon \bG \rightarrow \bG$ be the 
split Frobenius map such that $F_q(\bB)=\bB$ and $F_q(t)=t^q$ for all
$t \in \bT_0$. Let $\gamma \colon \bG\rightarrow \bG$ be the non-trivial 
graph automorphism such that $\gamma\circ F_q=F_q\circ \gamma$ and 
$\gamma(\bT_0)=\bT_0$, $\gamma(\bB)=\bB$. Then we consider the following 
two cases for a Frobenius map $F\colon \bG\rightarrow\bG$.
\begin{itemize}
\item[(a)] $F:=F_q$; in this case, we have $\bG^F=E_6(q)_{\text{ad}}\;$
(split type).
\item[(b)] $F:=F_q\circ\gamma=\gamma\circ F_q$; in this case, we have
$\bG^F={^2\!E}_6(q)_{\text{ad}}\;$ (twisted type).
\end{itemize}
Note that, in both cases, $\gamma$ induces an automorphism of the finite 
group $\bG^F$, which we denote by the same symbol. We fix a surjective 
homomorphism $\pi \colon \tbG \rightarrow \bG$ as in the previous section. 
(An explicit root datum for $\tbG$ can be found in \cite[\S 2]{hlm}.) Note 
that $\bK:=\ker(\pi_{\text{der}})=Z(\tbG_{\text{der}})$ has order~$3$ in 
this case. Hence, if $\tF$ does not act trivially on~$\bK$, then $\bK^\tF=
\{1\}$ and $\bG^F=\pi(\tbG_{\text{der}}^\tF)$; otherwise, we have 
$\bG^F/\pi(\tbG_{\text{der}}^\tF)\cong \bK$. The explicit description of 
$Z(\tbG_{\text{der}})$ in \cite[Example~1.5.6]{gema} shows the following. 
The map~$\tF$ acts trivially on $\bK$ precisely when $3\mid q-1$ in 
case~(a), or when $3\mid q+1$ in case~(b); furthermore, if $\tF$ does not 
act trivially on $\bK$, then $\tF(z)=z^{-1}$ for $z\in \bK$.

By Mizuno \cite{Miz} (see also L\"ubeck's tables \cite{Lue07}), there is 
a unique conjugacy class $\bC_0$ of semisimple elements of $\bG$ such that, 
for $s\in\bC_0$, we have $|A_\bG(s)|=3=|\bK|$ and $C_\bG^\circ(s)$ has a 
root system of type $A_2{+}A_2{+}A_2$. Note that the uniqueness implies
that $\bC_0$ is $F$-stable and $\bC_0^{-1}=\bC_0=\gamma(\bC_0)$. The aim 
of this section is to show how the character values $\rho(g)$ can be 
determined, for any $\rho \in \Irr(\bG^F)$ and any $g\in \bG^F$ with 
semisimple part in~$\bC_0$. 

First we deal with the case where $g$ is not a regular element.

\begin{abs} \label{lem51} Let $g\in \bG^F$ and assume that $g=su=us$ where
$s\in \bC_0^F$ and $u\in C_\bG^\circ(s)^F$ is unipotent, but not regular 
unipotent. Let $C$ be the $\bG^F$-conjugacy class of $g$. We claim that 
the indicator function $\varepsilon_C$ is a uniform function. This is seen
as follows. Let $\tilde{g}\in \tbG^\tF$ be such that $\pi(\tilde{g})=g$. 
Write $\tilde{g}=\tilde{s}\tilde{u}$ where $\tilde{s} \in \tbG^\tF$ is 
semisimple and $\tilde{u}\in \tbG^\tF$ is unipotent; here, $\tilde{\bH}:=
C_\tbG(\tilde{s})$ is connected. We claim that $C_\tbG(\tilde{g})=
C_{\tilde{\bH}}(\tilde{u})$ is connected. Note that this information 
can not be derived from Mizuno \cite{Miz}, since he classifies the 
conjugacy classes for $E_6$ of simply connected type, and he only gives 
the number of unipotent classes in the centraliser of a semisimple element.
So some additional computations are required in order to obtain 
$C_{\tilde{\bH}}(\tilde{u})$; but this is part of Michel's {\sf CHEVIE} 
\cite{gap3jm}. The following sequence of functions yields 
the required information on $C_{\tilde{\bH}}(\tilde{u})$: 



{\small 
\begin{verbatim}
   julia> W=rootdatum("CE6")                      # root datum of tilde{G}
   julia> W.rootdec[36]                           
   1 2 2 3 2 1                                    # 36: hightest root
   julia> H=reflection_subgroup(W,[1,2,3,5,6,36]) # subgroup of type A2+A2+A2
   julia> UnipotentClasses(H)                     # for any good prime p
   julia> UnipotentClasses(H,2)                   # for p=2
\end{verbatim}}

\noindent The output of the last two functions shows that only the regular 
unipotent elements in~$\tilde{\bH}$ have a non-connected centraliser. Now 
Lemma~\ref{scad3b} implies that $A_\bG(g)^F$ is isomorphic to a subgroup 
of~$\bK^\tF$. Hence, if $\tF$ does not act trivially on $\bK$, then 
$A_\bG(g)^F=\{1\}$ and so $\varepsilon_C$ is a uniform function; see 
Example~\ref{exp00}. On the other hand, if $\tF$ acts trivially on $\bK$, 
then Corollary~\ref{scad3} shows that, again, $\varepsilon_C$ is a uniform 
function. Hence, in both cases, the values of $\rho\in\Irr(\bG^F)$ on 
$g=su$ can be computed using the formula in Example~\ref{unif1}.
\end{abs}

\begin{abs} \label{lem52} From now on let $\Sigma$ be the $F$-stable 
$\bG$-conjugacy class of all elements of the form $g=su$ where $s\in\bC_0$ 
and $u\in C_\bG^\circ(s)$ is regular unipotent. Since $\bC_0^{-1}=\bC_0=
\gamma(\bC_0)$ (as noted above) and since the regular unipotent elements 
in a connected reductive group are all conjugate, it follows that 
$\Sigma^{-1}=\Sigma=\gamma(\Sigma)$. By the discussion in \cite[5.2, 
5.6]{pamq}, we can fix a representative $g_0\in \Sigma^F$ such that 
$g_0=\pi(\tilde{g}_0)$ where $\tilde{g}_0 \in \tbG_{\text{der}}^\tF$ 
and $\tilde{g}_0$ is $\tbG^\tF$-conjugate to~$\tilde{g}_0^{-1}$. 
Consequently, $g_0$ will be $\bG^F$-conjugate to~$g_0^{-1}$. By 
\cite[Cor.~20.4]{L2d} (see also \cite[4.6]{S3}), we have 
\[A_\tbG(\tilde{g}_0)=3\quad\mbox{and}\quad A_\bG(g_0)=C_3 \times C_3'
\quad\mbox{where $C_3$, $C_3'$ are cyclic of order~$3$}.\]
Here, $C_3=\langle \bar{g}_0\rangle$ is generated by the image of $g_0$ 
in $A_\bG(g_0)$; to single out a definite choice for $C_3'$ is a bit more 
tricky; we will come back to this issue in Lemma~\ref{choosea} below. In
any case, $F$ acts trivially on $C_3$ and so either $A_\bG(g_0)^F=C_3$ 
or $A_\bG(g_0)^F=A_\bG(g_0)=C_3\times C_3'$. 

The particular significance of the class $\Sigma$ is as follows. By 
\cite[Cor.~20.4]{L2d} (and its proof), there are six cuspidal character 
sheaves on $\bG$; they all have support given by $\Sigma$ and they correspond 
to those linear characters $\lambda\colon A_\bG(g_0)\rightarrow \C^\times$ 
that are non-trivial on~$C_3$. For the convenience of the reader, we display
the values of these six linear characters in Table~\ref{chrfe6}, where
$a$ is any element in $A_\bG(g_0)\setminus C_3$ and $C_3'=\langle a
\rangle$. 

Let $A$ be one of the six cuspidal character sheaves on $\bG$. Then $F^*A
\cong A$ if and only if the corresponding linear character of $A_\bG(g_0)$ 
is invariant under the action of $F$ on $A_\bG(g_0)$. Two out of the six 
cuspidal character sheaves are unipotent. By Shoji \cite[4.6, 4.8]{S3}, 
they are $F$-invariant, regardless of whether $F$ is the split Frobenius 
map or not.
\end{abs}

\begin{table}[htbp] \caption{Linear characters of $A_\bG(g_0)$ that are
non-trivial on $C_3$} \label{chrfe6}
\begin{center} $\renewcommand{\arraystretch}{1.1} \begin{array}{lccccccccc} 
\hline & 1 & \bar{g}_0 & \bar{g}_0^2  & a & \bar{g}_0a&
\bar{g}_0^2a&a^2&\bar{g}_0a^2&\bar{g}_0^2a^2\\\hline 
\psi_1 & 1 & \theta & \theta^2 & 1 & \theta & \theta^2 & 1 & \theta
& \theta^2\\ 
\psi_2=\overline{\psi}_1 & 1 &  \theta^2 & \theta & 1 & \theta^2 & 
\theta & 1 & \theta^2 & \theta\\ 
\psi_3 & 1 & \theta & \theta^2 & \theta & \theta^2 & 
1 & \theta^2 & 1 & \theta\\ 
\psi_4=\overline{\psi}_3 & 1 & \theta^2 & \theta & \theta^2 & \theta & 
1 & \theta & 1 & \theta^2\\ 
\psi_5 & 1 & \theta & \theta^2 & \theta^2& 1 & \theta & \theta & 
\theta^2 & 1 \\ 
\psi_6=\overline{\psi}_5 & 1 & \theta^2 & \theta & \theta& 
1 & \theta^2 & \theta^2 & \theta & 1 \\ \hline & \multicolumn{9}{c}{
\mbox{\footnotesize (Here, $1\neq \theta\in \C$ is a fixed third root 
of unity)}}
\end{array}$
\end{center}
\end{table}

If $F$ does not act trivially on $A_\bG(g_0)$, then $A_\bG(g_0)^F=C_3$
and so there will only be two linear characters in Table~\ref{chrfe6} that
are invariant under the action of $F$; these will correspond to the two
unipotent cuspidal character sheaves (and these are the only $F$-invariant
cuspidal character sheaves). On the other hand, if $F$ acts trivially on 
$A_\bG(g_0)$, then we will have to specify which of the six linear
characters in Table~\ref{chrfe6} correspond to the two unipotent cuspidal
character sheaves; this is related to the choice of the subgroup~$C_3'$ 
such that $A_\bG(g_0)=C_3 \times C_3'$.

\begin{prop} \label{e6m1} Assume that $F$ does not act trivially on 
$A_\bG(g_0)$ (where $g_0\in \Sigma^F$ is $\bG^F$-conjugate to~$g_0^{-1}$, 
as above). Then $A_\bG(g_0)^F=C_3$ and $C_3=\{1,\bar{g}_0,\bar{g}_0^2\}$ 
is a set of representatives of the $F$-conjugacy classes of $A_\bG(g_0)$. 
The characteristic functions of the two ($F$-invariant) unipotent cuspidal
character sheaves (normalised as in \ref{cuspcs1}) are given by the 
following table:
\[\renewcommand{\arraystretch}{1.1}  \begin{array}{cccc} \hline  & 
C_1 & C_{\bar{g}_0} & C_{\bar{g}_0^2} \\ \hline \chi_1 & q^3 & q^3\theta 
& q^3 \theta^2\\ \chi_2 & q^3 & q^3 \theta^2 & q^3 \theta\\ \hline
\end{array}\qquad\qquad \begin{array}{l} \mbox{where $1\neq \theta \in 
\C$ is a fixed third} \\ \mbox{root of unity (as in Table~\ref{chrfe6})}.
\end{array}  \]
\end{prop}

\begin{proof} By Example~\ref{exp00}(a), the number of $F$-conjugacy
classes of $A_\bG(g_0)$ is given by $|A_\bG(g)^F|=|C_3|=3$. Hence, it
is sufficient to show that the elements of $C_3$ belong to different
$F$-conjugacy classes. To see this, let $\psi\colon A_\bG(g_0)\rightarrow 
\C^\times$ be a linear character corresponding to one of the two unipotent 
cuspidal character sheaves; since the latter is $F$-invariant (as noted 
above), the character $\psi$ must be constant on the $F$-conjugacy classes 
of $A_\bG(g_0)$. (This immediately follows from the fact that $\psi$ is
a group homomorphism that is invariant under the action of $F$.) Now
we observe that each of the six linear characters in Table~\ref{chrfe6} has 
pairwise different values on $1,\bar{g}_0,\bar{g}_0^2$. In particular,
the same must be true for $\psi$. Hence, the elements $1,\bar{g}_0,
\bar{g}_0^2$ must belong to different $F$-conjugacy classes. Finally,
the values of $\chi_1,\chi_2$ are obtained by the formula in
\ref{cuspcs1}(b); note that $\dim \bG-\dim \Sigma=6$ since $\Sigma$
is a conjugacy class of regular elements in~$\bG$.
\end{proof}

\begin{lem} \label{e6m2} The map $\tF$ acts trivially on $\bK$ if and only 
if $F$ acts trivially on $A_\bG(g_0)$. 
\end{lem}

\begin{proof} To see this, we consider the injective homomorphism 
$C_\bG(g_0)/\pi(C_\tbG(\tilde{g}_0))\hookrightarrow \bK$ in
Remark~\ref{scad3c}.  By its construction, it is compatible with the 
action of $F$ on the left side and the action of $\tF$ on the right side.
Furthermore, $3=|A_\tbG(\tilde{g}_0)|=|\pi(C_\tbG(\tilde{g}_0)):
C_\bG^\circ(g_0)|$ and so $C_\bG(g_0)/\pi(C_\tbG(\tilde{g}_0))\neq \{1\}$. 
Hence, if $\tF$ does not act trivially on $\bK$, then $F$ will not act 
trivially on $A_\bG(g_0)$ either. Conversely, assume that $\tF$ acts 
trivially on $\bK$. Then, since $|A_\bG(s_0)|=3=|\bK|$, the set $\bC_0^F$ 
will split into three conjugacy classes in $\bG^F$; see Example~\ref{lem00a}.
Now assume, if possible, that $F$ does not act trivially on $A_\bG(g_0)$. 
Then, by Proposition~\ref{e6m1}, we have $\Sigma^F=C_1\cup C_{\bar{g}_0}
\cup C_{\bar{g}_0^2}$. So Lemma~\ref{exp00a} shows that the semisimple 
parts of all elements in $\Sigma^F$ form one $\bG^F$-conjugacy 
class~---~contradiction, since $\bC_0^F$ precisely is that set of 
semisimple parts of elements in $\Sigma^F$; see Remark~\ref{scad1a}.
\end{proof}

\begin{abs} \label{lem53} Assume now that $F$ acts trivially on 
$A_\bG(g_0)$, where $g_0\in \Sigma^F$ is $\bG^F$-conjugate to~$g_0^{-1}$
(see \ref{lem52}). Let $C_3=\langle \bar{g}_0\rangle$ and choose any 
element $a\in A_\bG(g_0)\setminus C_3$. Then, setting $C_3'=\langle 
a\rangle$, we have $A_\bG(g_0)=C_3\times C_3'= \langle \bar{g}_0 \rangle 
\times \langle a \rangle$ and $\Sigma^F$ splits into $9$ conjugacy classes 
in $\bG^F$:
\begin{equation*}
\Sigma^F = C_1\cup C_{\bar{g}_0} \cup C_{\bar{g}_0^2}  \cup 
C_a\cup C_{\bar{g}_0a} \cup C_{\bar{g}_0^2a}  \cup
C_{a^2}\cup C_{\bar{g}_0a^2} \cup C_{\bar{g}_0^2a^2}\tag{a}
\end{equation*}
where $g_0\in C_1=C_1^{-1}$. By Lemma~\ref{e6m2}, $\tF$ also acts trivially
on $\bK$. So, since $|A_\bG(s_0)|=3=|\bK|$, the set $\bC_0^F$ splits into 
$3$ conjugacy classes of (semisimple elements of) $\bG^F$, and all these
elements arise as semisimple parts of elements of $\Sigma^F$ (see 
Remark~\ref{scad1a} and Example~\ref{lem00a}). Let $s_0,s_1,s_2$ be 
representatives of these classes in~$\bC_0^F$. Taking into account 
Lemma~\ref{exp00a}, we see that the notation can be fixed such that, 
for $i=0,1,2$, the semisimple parts of the elements in the subset 
\begin{equation*}
\Sigma_i^F:=C_{a^i} \cup C_{\bar{g}_0a^i}\cup C_{\bar{g}_0^2a^i}\subseteq
\Sigma^F\tag{b}
\end{equation*}
are $\bG^F$-conjugate to $s_i$; thus, $\Sigma_i^F=\Sigma_{s_i}^F$ with
the notation of Proposition~\ref{scad1}. Since $s_0,s_1,s_2$ are also 
representatives of the $3$ cosets of $\pi(\tbG_{\text{der}}^\tF)$ in 
$\bG^F$ (see again Example~\ref{lem00a}), and since $g_0\in \pi
(\tbG_{\text{der}}^\tF)$, we can further fix the notation such that 
$s_0\in \pi(\tbG_{\text{der}}^\tF)$. Consequently, we have 
\begin{equation*}
\Sigma_0^F\subseteq \pi(\tbG_{\text{der}}^\tF), \qquad 
\Sigma_1^F\subseteq s_1\pi(\tbG_{\text{der}}^\tF), \qquad
\Sigma_2^F\subseteq s_2\pi(\tbG_{\text{der}}^\tF).\tag{c}
\end{equation*}
With these preparations, we can now state:
\end{abs}

\begin{prop} \label{choosea} Assume that $F$ acts trivially on $A_\bG(g_0)$,
where $g_0\in \Sigma^F$ is such that $g_0^{-1}$ is $\bG^F$-conjugate to $g_0$;
also recall that $\gamma(\Sigma)=\Sigma$ (see \ref{lem52}, \ref{lem53}). Then
an element $a\in A_\bG(g_0)\setminus C_3$ can be chosen such that the action 
of $\gamma\in \operatorname{Aut}(\bG^F)$ on $\Sigma^F$ is given by 
\[ \gamma(C_{\bar{g}_0^i})=C_{\bar{g}_0^i} \qquad\mbox{and}\qquad
 \gamma(C_{\bar{g}_0^ia})=C_{\bar{g}_0^ia^2} \qquad\mbox{for $i=0,1,2$}.\]
For any such choice of $a\in A_\bG(g_0)\setminus C_3$, the linear
characters $\lambda \colon A_\bG(g_0) \rightarrow \C^\times$ corresponding 
to the two ($F$-invariant) unipotent cuspidal character sheaves on $\bG$ 
are precisely those that are non-trivial on $C_3$ and trivial on 
$C_3'=\langle a\rangle$ (see Table~\ref{chrfe6}). 
\end{prop}

\begin{proof} Since $\gamma(\Sigma)=\Sigma$, and since $\gamma$ and $F$ 
commute with each other, the automorphism~$\gamma$ will permute the 
$\bG^F$-conjugacy classes inside $\Sigma^F$. Let $A_1,A_2$ be the two 
unipotent cuspidal character sheaves on $\bG$ and $\chi_1,\chi_2\in
\CF(\bG^F\mid \Sigma)$ the corresponding characteristic functions, 
normalised as in \ref{cuspcs1}(b). By the main result of Shoji \cite{S2}, 
\cite{S3}, the functions $\chi_1,\chi_2$ are linear combinations of 
unipotent characters. The form of these linear combinations shows that 
$\chi_2=\overline{\chi}_1$ (complex conjugate); see also \cite[5.4, 
5.6]{pamq}. Thus, $A_1,A_2$ correspond to complex conjugate linear 
characters of $A_\bG(g_0)$ in Table~\ref{chrfe6}, which must be either 
$\{\psi_1,\psi_2\}$ or $\{\psi_3,\psi_4\}$ or $\{\psi_5,\psi_6\}$. But 
that table also shows that, replacing an initially chosen element
$a\in A_\bG(g_0) \setminus C_3$ by $\bar{g}_0^ia$ for a suitable $i\in 
\{0,1,2\}$ if necessary, we may assume without loss of generality that
the two linear characters are $\psi_1,\psi_2$; note that these are trivial 
on~$C_3'$. The values of $\chi_1,\chi_2$ are now given as follows.
\begin{center} $\renewcommand{\arraystretch}{1.1} \begin{array}{lccccccccc} 
\hline & C_1 & C_{\bar{g}_0} & C_{\bar{g}_0^2}  & C_a & C_{\bar{g}_0a}&
C_{\bar{g}_0^2a}&C_{a^2}&C_{\bar{g}_0a^2}&C_{\bar{g}_0^2a^2}\\\hline 
\chi_1 & q^3 & q^3 \theta & q^3 \theta^2 & q^3 & q^3 \theta & 
q^3 \theta^2 & q^3 & q^3 \theta & q^3 \theta^2\\ 
\chi_2=\overline{\chi}_1 & q^3 & q^3 \theta^2 & q^3 \theta & q^3 & 
q^3 \theta^2 & q^3 \theta & q^3 & q^3 \theta^2 & q^3 \theta\\ \hline
\end{array}$
\end{center}
(Note again that $\dim \bG-\dim \Sigma=6$). Now it is known that all 
unipotent characters of $\bG^F$ are invariant under $\gamma\in \mbox{Aut}
(\bG^F)$; see \cite[Theorem~4.5.11]{gema}. Hence, we must also have 
\[ \chi_1(g)=\chi_1(\gamma(g))\qquad \mbox{for all $g\in \bG^F$}.\] 
Now $\gamma$ leaves $\pi(\tbG_{\text{der}}^\tF)$ invariant. Hence, since 
$\Sigma_0^F=C_1\cup C_{\bar{g}_0}\cup C_{\bar{g}_0^2}\subseteq \pi
(\tbG_{\text{der}}^\tF)$ (see \ref{lem53}(c)), we have $\gamma(\Sigma_0^F)
=\Sigma_0^F$. Looking at the values of $\chi_1$, and taking into account 
the invariance of $\chi_1$ under~$\gamma$, we conclude that $\gamma
(C_{\bar{g}_0^i})=C_{\bar{g}_0^i}$ for $i=0,1,2$. By a similar argument, 
$\gamma(\Sigma_1^F)$ must be equal to $\Sigma_1^F$ or $\Sigma_2^F$; 
consequently (again by looking at the values of $\chi_1$), we have 
\[\gamma(C_{\bar{g}_0^ia})=C_{\bar{g}_0^ia} \quad (i=0,1,2)\qquad\mbox{or} 
\qquad \gamma(C_{\bar{g}_0^ia})=C_{\bar{g}_0^ia^2} \quad (i=0,1,2).\]
Assume, if possible, that the first possibility is true, that is,
$\gamma(C_{\bar{g}_0^ia})=C_{\bar{g}_0^ia}$ for $i=0,1,2$.  Then all
$\bG^F$-conjugacy classes contained in $\Sigma^F$ are fixed by $\gamma$,
and so the characteristic functions of all six cuspidal character sheaves
will be invariant under~$\gamma$. This would imply that, if $A$ is any
of those six cuspidal character sheaves, then $\gamma^*A\cong A$, 
contradiction to \cite[Cor.~20.4]{L2d} (and its proof). So the second
possibility must be true.

Thus, indeed, $a\in A_\bG(g_0)$ can be chosen such that the action of 
$\gamma$ on the $\bG^F$-conjugacy classes contained in $\Sigma^F$ is as 
stated. (If $F=F_q$ is split and $3\mid q-1$, then this also follows 
immediately from the proof of \cite[Cor.~20.4]{L2d}.) But then the 
statement concerning the linear characters corresponding to the two 
unipotent cuspidal character sheaves follows for any such choice of~$a$, 
again by inspection of the values of $\chi_1$ and $\chi_2$. 
\end{proof}

\begin{abs} \label{e6m3} Let $\Sigma$ and $g_0\in \Sigma^F$ be as in
\ref{lem52}. We can now address the problem of finding a suitable basis 
of the subspace $\CF(\bG^F\mid \Sigma)$, satisfying the requirements 
$(\Sigma)_1$, $(\Sigma)_2$ in Section~\ref{subsp}. Let $\chi_0$ be the
indicator function of $\Sigma^F$; this is a uniform function and we can 
explicitly write it as a linear combination of $\Irr(\bG^F)$ (see~\ref{unif} 
and also Remark~\ref{zent0} below.) Let $\chi_1,\chi_2$ be the characteristic 
functions of the two ($F$-invariant) unipotent cuspidal character sheaves;
explicit expressions as linear combinations of unipotent characters of $\bG^F$
can be found in \cite[Prop.~5.5]{pamq} (split type) and \cite[Prop.~5.7]{pamq} 
(twisted type).

(a) Assume first that $F$ does not act trivially on $A_\bG(g_0)$; then
$\dim \CF(\bG^F\mid \Sigma)=|A_\bG(g_0)^F|=3$. Using
Proposition~\ref{e6m1}, a basis of $\CF(\bG^F\mid \Sigma)$ is 
given in Table~\ref{chrfe6b}.

(b) Now assume that $F$ acts trivially on $A_\bG(g_0)$; then $\dim 
\CF(\bG^F\mid \Sigma)=|A_\bG(g_0)^F|=9$. Recall from \ref{lem53} 
that $s_0,s_1,s_2$ are representatives of the $\bG^F$-conjugacy classes 
contained in~$\bC_0^F$, and that these elements also form a set of 
representatives for the $3$ cosets of $\pi(\tbG_{\text{der}}^\tF)$ in 
$\bG^F$, where $s_0$ represents the trivial coset. Hence, for $i=0,1,2$, 
we have a unique linear character $\lambda_i \colon \bG^F\rightarrow 
\C^\times$ such that $\pi(\tbG_{\text{der}}^\tF)\subseteq \ker(\lambda_i)$ 
and 
\[ \lambda_i(s_0)=1, \qquad \lambda_i(s_1)=\theta^i, \qquad \lambda_i(s_2)=
\theta^{-i}.\] 
We apply the construction in Proposition~\ref{scad1} to $\chi_0,\chi_1,
\chi_2$. This yields class functions $\chi_0^{(s_i)}$, $\chi_1^{(s_i)}$, 
$\chi_2^{(s_i)}$ for $i=0,1,2$ which form a basis of $\CF(\bG^F\mid 
\Sigma)$; see Table~\ref{chrfe6b}. As already 
mentioned, explicit expressions for $\chi_1,\chi_2$ as linear combinations 
of unipotent characters of~$\bG^F$ are known. This yields explicit 
expressions of $\lambda_i\cdot \chi_1$ and $\lambda_i \cdot \chi_2$ as 
linear combinations of $\Irr(\bG^F)$. (See Remark~\ref{zent3} below for 
more details.) Consequently, we also obtain expressions of $\chi_1^{(s_i)}$ 
and $\chi_2^{(s_i)}$ as linear combinations of $\Irr(\bG^F)$, for $i=0,1,2$.
(We note that, for $i=1,2$, the functions $\lambda_i\cdot \chi_1$ and 
$\lambda_i \cdot \chi_2$ can also be identified with characteristic functions
of the four non-unipotent cuspidal character sheaves on $\bG$.)
\end{abs}


\begin{table}[htbp] \caption{Bases of $\CF(\bG^F\mid \Sigma)$} 
\label{chrfe6b}
\begin{center}
$\renewcommand{\arraystretch}{1.1} \renewcommand{\arraycolsep}{3pt}
\begin{array}{cccc} \multicolumn{4}{l}{\mbox{\small $|A_\bG(g_0)^F|=3:$}} \\ 
\hline & C_1 & C_{\bar{g}_0} & C_{\bar{g}_0^2} \\ \hline \chi_0 & 1 & 1 & 1\\ 
\chi_1 & q^3 & q^3\theta & q^3 \theta^2\\ \chi_2 & q^3 & q^3 \theta^2 
& q^3 \theta\\ \hline \\\\\\\\\\\\\end{array}\qquad \quad
\begin{array}{lccccccccc} \multicolumn{10}{l}{\mbox{\small $|A_\bG(g_0)^F|
=9:$}}\\ \hline & C_1 & C_{\bar{g}_0} & C_{\bar{g}_0^2}  & C_a &
C_{\bar{g}_0a}& C_{\bar{g}_0^2a}&C_{a^2}&C_{\bar{g}_0a^2}&C_{\bar{g}_0^2a^2}
\\\hline \chi_0^{(s_0)} & 1 & 1 & 1 & 0 & 0 & 0 & 0 & 0 & 0\\
\chi_1^{(s_0)} & q^3 & q^3 \theta & q^3 \theta^2 & 0 & 0 & 0 & 0 & 0 & 0 \\
\chi_2^{(s_0)} & q^3 & q^3 \theta^2 & q^3 \theta & 0 & 0 & 0 & 0 & 0 & 0 \\
\chi_0^{(s_1)} & 0 & 0 & 0 & 1 & 1 & 1 & 0 & 0 & 0\\
\chi_1^{(s_1)} & 0 & 0 & 0 & q^3 & q^3 \theta & q^3 \theta^2 & 0 & 0 & 0 \\ 
\chi_2^{(s_1)} & 0 & 0 & 0 & q^3 & q^3 \theta^2 & q^3 \theta & 0 & 0 & 0 \\ 
\chi_0^{(s_2)} & 0 & 0 & 0 & 0 & 0 & 0 & 1 & 1 & 1\\
\chi_1^{(s_2)} & 0 & 0 & 0 & 0 & 0 & 0 & q^3 & q^3 \theta & q^3 \theta^2 \\
\chi_2^{(s_2)} & 0 & 0 & 0 & 0 & 0 & 0 & q^3 & q^3 \theta^2 & q^3 \theta \\
\hline \end{array}$
\end{center}
\end{table}


As discussed in Section~\ref{subsp}, the above information about a 
basis of $\CF(\bG^F\mid \Sigma)$ allows us to determine the values
of all irreducible characters of $\bG^F$ on elements in $\Sigma^F$.

\section{Groups of adjoint type $E_6$ in characteristic~$2$} \label{finale6}

Let $\bG$ be simple, adjoint of type $E_6$ and $F\colon \bG\rightarrow \bG$ 
be a Frobenius map. We shall also fix a group $\bG^*$ dual to~$\bG$, with 
dual Frobenius map $F^* \colon \bG^* \rightarrow \bG^*$. (See 
\cite[\S 1.5]{gema} for general background.) Since $\bG$ is simple, adjoint 
of type $E_6$, the group $\bG^*$ is simple, simply connected of type~$E_6$. 

By Lusztig's general theory, $\Irr(\bG^F)$ is the disjoint union of 
``series'' $\cE(\bG^F,s)$, where $s$ runs over a set of representatives 
of the conjugacy classes of semisimple elements of~${\bG^*}^{F^*}$ (see 
\cite[\S 2.6]{gema}). As a first, and crucial step, we shall now consider 
the unipotent characters of $\bG^F$, which are the characters in the
series $\cE(\bG^F,1)$. A character $\rho\in \Irr(\bG^F)$ is unipotent 
if and only if $\langle R_{\bT,1}^\bG,\rho\rangle_{\bG^F}\neq 0$ for some
$F$-stable maximal torus $\bT\subseteq \bG$. There are $30$ unipotent 
characters, both for $\bG^F=E_6(q)$ and for $\bG^F={^2\!E}_6(q)$. Tables 
with their degrees (and further information) can be found in 
\cite[p.~480/481]{Ca2} and \cite[p.~363]{LuB}. 

In order to explain how the values of the unipotent characters can be 
determined, we shall assume that the following information is available.
\begin{itemize}
\item[{\bf (A1)}] A parametrisation of the conjugacy classes of $\bG^F$; 
\item[{\bf (A2)}] the values of $R_{\bT,1}^\bG$ for all $F$-stable
maximal tori $\bT\subseteq \bG$ (up to $\bG^F$-conjugacy);
\item[{\bf (A3)}] for every regular $\bL\subsetneqq \bG$, the values 
$\psi(u)$ for $\psi \in \cE(\bL^F,1)$ and $u\in \bL^F$ unipotent.
\end{itemize}

\begin{rem} \label{zent0} Assume that $f\in \CF(\bG^F)$ is a uniform
function. Using the formula in~\ref{unif}, we can express $f$ as a linear
combination of the generalised characters $R_{\bT,\theta}^\bG$ where
the coefficients involve the inner products $\langle f,R_{\bT,\theta}^\bG
\rangle_{\bT^F}$. Now the decomposition of each $R_{\bT,\theta}^\bG$ as
a linear combination of $\Irr(\bG^F)$ is explicitly known by the 
Main Theorem~4.23 of Lusztig \cite{LuB}. In order to compute $\langle 
f,R_{\bT,\theta}^\bG\rangle_{\bG^F}$, we need to know the values of
$R_{\bT,\theta}^\bG$. Hence, once {\bf (A1)}, {\bf (A2)} are available, 
we can explicitly write $f$ as a linear combination of $\Irr(\bG^F)$.
\end{rem}

\begin{abs} \label{ab1} {\bf About (A1)}. The conjugacy classes of
semisimple elements of any finite group of Lie type can be classified by 
a standard procedure; see Deriziotis \cite{Dere}, Fleischmann--Janiszczak 
\cite{FlJa} and further references there. For $\bG$ of type $E_6$, this 
was first done by Mizuno \cite{Miz}; explicit tables for this case (and 
also other groups of exceptional type) are made available by L\"ubeck 
\cite{Lue07}. Now let $s\in \bG^F$ be semisimple. Then one needs to
classify the conjugacy classes of unipotent elements in $C_\bG(s)^F$, 
and this is somewhat more tricky. In our case, it is convenient to pass
to the group $\tbG$ and consider the homomorphism $\pi \colon \tbG
\rightarrow \bG$ with $\pi(\tbG^\tF)=\bG^F$ (as in Section~\ref{scad}). In 
\cite[\S 2]{hlm} and \cite[\S 3]{hl} it is reported that parametrisations
of the conjugacy classes of $\tbG^\tF$ have been computed using (the older
version of) {\sf CHEVIE} \cite{chevie}. Information about the unipotent 
classes in the centralisers of semisimple elements of~$\tbG$ can now also be
obtained using Michel's more recent version of {\sf CHEVIE} \cite{gap3jm}, 
as already mentioned (and used) in~\ref{lem51}; from this one can derive the 
conjugacy classes of~$\bG^F$. Formulae for the total number of conjugacy
classes of~$\bG^F$ can be found in \cite[\S 4]{FlJa} (they depend on
$q \bmod 6$).
\end{abs}

\begin{abs} \label{ab2} {\bf About (A2)}. There is a character formula 
which reduces the computation of the values of $R_{\bT,\theta}^\bG$ (for 
any $\theta \in \Irr(\bT^F)$) to the computation of the Green functions 
of $\bG^F$, and the Green functions of subgroups of the form $C_\bG^\circ 
(s)^F$ where $s\in \bG^F$ is semisimple; see \cite[Theorem~2.2.16]{gema}. 
The further issues in the evaluation of $R_{\bT,\theta}^\bG$ are discussed 
in \cite[\S 3]{pamq}. If $\theta=1$ is the trivial character, then the 
whole procedure simplifies drastically.

In the case of primary interest to us, that is, $p=2$, the Green functions 
for $\bG^F$ itself have been computed by Malle \cite{Mal1}. Nowadays, one 
can re-produce these results somewhat more systematically by the methods 
in \cite{ekay} (especially for $p=2$); these methods also apply to all
subgroups $C_\bG^\circ(s)^F$ as above. The remarks in \cite[\S 2]{hlm} and 
\cite[\S 3]{hl} indicate that Frank L\"ubeck has already computed explicit 
tables with the values $R_{\bT,1}^\bG(g)$ for all $F$-stable maximal
tori $\bT\subseteq \bG$ and all $g\in \bG^F$. 
\end{abs}

\begin{abs} \label{ab13} {\bf About (A3)}. First note that, since $\bG$ has 
a connected center, the same is true for~$\bL$ as well. Furthermore, the 
possibilities for the root system of $\bL$ are as follows. Either this is 
a direct sum of root systems of type $A_m$ (for various $m$), or it is a 
root system of type~$D_4$ or of type~$D_5$. In the first case, we can
use Example~\ref{unif1} and Remark~\ref{zent0}. For the second case,
explicit tables with the values of the unipotent characters (both for 
$D_4(q)$ and for ${^2\!D}_4(q)$) are contained in the library of (the 
{\sf MAPLE} part of) {\sf CHEVIE} \cite{chevie}. (Proposition~\ref{d4split}
provides a proof that those tables are correct.) Finally, assume that we
are in the third case, that is, $\bL$ of type $D_5$. Then the information 
about the values of the unipotent characters of~$\bL^F$ on unipotent
elements does not yet seem to be available; the required arguments
are contained in an appendix due to J. Hetz (see Section~\ref{appdx}).
\end{abs}


Assume from now on that $p=2$ and that the information in {\bf (A1)}, 
{\bf (A2)}, {\bf (A3)} indeed is available. In order to proceed, we need 
to know the possible types of the centralisers of semisimple elements 
in~$\bG$.

\begin{abs} \label{ab0} Recall that we assume that $p=2$. Let $s\in \bG^F$ 
be semisimple and $\bH_s:=C_\bG^\circ(s)$. By inspection of the list in 
Theorem~3.8 of Mizuno \cite{Miz} (see also L\"ubeck's tables \cite{Lue07}), 
we see that $\bH_s$ is either a maximal torus, or $\bH_s$ is a regular 
subgroup with a root system of type 
\begin{gather*}
E_6,\;D_5,\;A_5,\;A_4{+}A_1,\;A_2{+}A_2{+}A_1,\;D_4,\;A_4,\;
A_3{+}A_1,\; A_2{+}A_2,\; \\ A_2{+}A_1{+}A_1,\; A_3,
\;A_2{+}A_1,\;A_1{+}A_1{+}A_1,\;A_2,\;A_1{+}A_1,\;A_1,
\end{gather*}
or $\bH_s$ has a root system of type $A_2{+} A_2{+} A_2$. (Note
that, in Mizuno's list, there are also subgroups with a root system of type
$A_5{+}A_1$, $A_3{+}A_1 {+}A_1,A_1{+}A_1{+}A_1{+}A_1$, but these do not
occur as connected centralisers of semisimple elements when $p=2$.) Thus,
we have the following special situation for our group $\bG$ (with $p=2$): 
Either $\bH_s=\bG$ or $\bH_s$ has a root system of type $A_2{+} A_2{+}A_2$, 
or $\bH_s \subsetneqq \bG$ is a regular subgroup as in~\ref{exp01}.
\end{abs}

Given the above information, we can now follow the strategy in 
\cite[\S 3]{unif4} to compute the values of the unipotent characters of
$\bG^F$ on all elements of $\bG^F$. We give just some more specific 
comments. Let $g\in \bG^F$ and write $g=su=us$ where $s\in \bG^F$ is 
semisimple and $u\in \bG^F$ is unipotent. As above let $\bH_s:=
C_\bG^\circ(s)$; then $u \in \bH_s^F$.

\begin{rem} \label{zent1} Assume that $\bH_s=\bG$, that is, $s=1$ and
$g=u$ is unipotent. Then the values of the unipotent characters of 
$\bG^F$ on~$g$ have been determined by Hetz \cite[Remark~4.1.27]{Het3}
(split type) and \cite[\S 8]{Het4} (twisted type).
\end{rem}

\begin{rem} \label{zent2} Assume that $\bH_s \subsetneqq \bG$ is a 
regular subgroup. Then we can use the information in {\bf (A3)} and
the formula in~\ref{exp01} to compute the values of the
unipotent characters of $\bG^F$ on~$g$, exactly as explained in 
\cite[3.4]{unif4}.
\end{rem}

\begin{rem} \label{zent3} Assume that $\bH_s$ has a root system of
type $A_2{+} A_2{+} A_2$. Thus, we are in the setting of
Section~\ref{e6p2}. Let $C$ be the $\bG^F$-conjugacy class of~$g$. For a
given unipotent character $\rho \in \cE(\bG^F,1)$, we then use the formula 
$\rho(g)=|C_\bG(g)^F|\langle \rho,\varepsilon_C\rangle_{\bG^F}$. 
If $u$ is not regular unipotent in $\bH_s$, then $\varepsilon_C$
is uniform (see \ref{lem51}); so we can use Remark~\ref{zent0} to evaluate
$\langle \rho,\varepsilon_C\rangle_{\bG^F}$. Now assume that $u$ is 
regular unipotent in~$\bH_s$; thus, $g=su \in \Sigma^F$ with $\Sigma$ as 
in~\ref{lem52}. Here, the indicator function $\varepsilon_C$ is not uniform,
but we can explicitly express $\varepsilon_C$ as a linear combination of the 
functions in Table~\ref{chrfe6b}. Hence, it remains to determine the inner
products of $\rho$ with those functions. Firstly, $\chi_0$ and 
$\chi_0^{(s_i)}$ ($i=0,1,2$) are uniform, so we can use Remark~\ref{zent0}.
Next, $\chi_1,\chi_2$ are (known) linear combinations of the unipotent 
characters of $\bG^F$; hence, we know the inner products of $\rho$ with 
$\chi_1$ and $\chi_2$. It remains to determine the inner products of 
$\rho$ with $\chi_1^{(s_i)}$ and $\chi_2^{(s_i)}$ for $i=0,1,2$. For this 
purpose, it is enough to consider the products $\lambda_i \cdot \chi_1$ 
and $\lambda_i \cdot \chi_2$, where $\lambda_i\colon \bG^F \rightarrow
\C^\times$ is a linear character such that $\pi(\tbG_{\text{der}}^\tF)
\subseteq \ker(\lambda_i)$, as in \ref{e6m3}(b). We claim that
\[ \langle \rho,\lambda_i \cdot \chi_1\rangle_{\bG^F}=\langle \rho,
\lambda_i \cdot \chi_2\rangle_{\bG^F}=0 \qquad \mbox{for $i=1,2$}.\]
This is seen as follows. By \cite[Prop.~2.5.20]{gema}, each linear character
$\lambda_i$ as above corresponds to an element $z_i\in Z(\bG^*)^{F^*}$; 
furthermore, multiplication with $\lambda_i$ defines a bijection 
$\cE(\bG^F,1)\stackrel{\sim}{\rightarrow} \cE(\bG^F,z_i)$. (This follows 
from \cite[Prop.~2.5.21]{gema}.) Now assume that $i\in \{1,2\}$. Then
$\lambda_i$ is non-trivial and, hence, $z_i\neq 1$. Consequently, $\lambda_i
\cdot \chi_1$ and $\lambda_i\cdot \chi_2$ are linear combinations of 
characters in the series $\cE(\bG^F,z_i)\neq \cE(\bG^F,1)$, and so the 
inner product of our unipotent character $\rho$ with $\lambda_i\cdot 
\chi_1$ and with $\lambda_i \cdot \chi_2$ must be~$0$. 
\end{rem}

\begin{rem} \label{zent4} In order to determine the values of the 
remaining, non-unipotent characters of $\bG^F$, one can proceed similarly 
as in \cite[\S 4]{unif4}, assuming that the following information is 
available.
\begin{itemize}
\item[{\bf (A5)}] A parametrisation of $\fX(\bG,F)$, modulo the natural 
action of $\bG^F$ by conjugation;
\item[{\bf (A6)}] parametrisations of the semisimple conjugacy classes 
of ${\bG^*}^{F^*}$;
\item[{\bf (A7)}] tables with the values $R_{\bT,\theta}^\bG(g)$ for 
all $(\bT,\theta) \in \fX(\bG,F)$ and all $g\in \bG^F$; 
\end{itemize}
\end{rem}
In any case, as already mentioned in the introduction, a lot of 
further~---~mainly computational~---~work is required in order to produce 
actual tables of character values for~$\bG^F$ but, at least, we hope to 
have given a clear strategy of how to obtain those tables.

\section{Appendix. Groups of type $D_5$ in characteristic~$2$ (by J.
Hetz)} \label{appdx}

Let $p=2$ and $\bG$ be simple of type $D_5$, so that we either have $\bG^F=D_5(q)$ 
(split type) or $\bG^F={^2\!D}_5(q)$ (twisted type). The aim of this 
appendix is to determine the values of the unipotent characters on the 
unipotent elements of $\bG^F$. (This information was needed in
Section~\ref{finale6}.) First of all, by \cite[Table~8.6a (p.~127)]{LiSe}, 
there are $16$ unipotent classes in~$\bG$ and $20$ unipotent classes in
$\bG^F$ (both for the split type and the twisted type). Hence, we have
$\dim\mbox{CF}_{\text{uni}}(\bG^F)=20$. By Proposition~\ref{clp2b}, there 
are also $20$ unipotent characters of $\bG^F$. So the issue is to determine
a certain $(20 \times 20)$-array of character values.
There are $18$ generalised characters $R_{\bT,1}^\bG$ (again both for the split type and the twisted type), and we obtain their decompositions into unipotent characters as explained in \ref{d4intro}. The values of their restrictions to the unipotent elements (that is, the Green functions) are known and can be computed using the methods in \cite{ekay}. In view of Proposition~\ref{clp2b}, it remains to define suitable linear combinations of unipotent characters to obtain two further class functions $f_1,f_2\in\CF(\bG^F)$ which are orthogonal to all $R_{\bT,1}^\bG$, and then compute the values of $f_1, f_2$ on unipotent elements.

\begin{abs} \label{D5a} (a) Assume that $\bG^F=D_5(q)$ (split type). The 
unipotent characters of $\bG^F$ are parametrised by symbols of rank $5$ 
and defect $d\equiv 0 \bmod 4$ (cf.\ \ref{d4intro}; there is no symbol of the form $(S,S)$ in the present situation). We define two class functions $f_1,f_2\in \mbox{CF}(\bG^F)$ by 
\begin{align*}
f_1&:=\tfrac12(\rho_{(02,14)}-\rho_{(12,04)}-\rho_{(01,24)}+\rho_{(0124,-)}), \\
f_2&:=\tfrac12(\rho_{(013,124)}-\rho_{(123,014)}-\rho_{(012,134)}+\rho_{(01234,1)}).
\end{align*}
As in \ref{d4intro}, one sees that $\langle f_i,R_{\bT,1}^\bG
\rangle_{\bG^F}=0$ for $i=1,2$ and for all $\bT$; also note that
$\langle f_1,f_2\rangle_{\bG^F}=0$. Hence, every unipotent 
character of $\bG^F$ is uniquely a linear combination of $f_1$, $f_2$
and the various $R_{\bT,1}^\bG$.

(b) Assume that $\bG^F={^2\!D}_5(q)$ (twisted type). The unipotent 
characters of $\bG^F$ are parametrised by symbols of rank $5$ and 
defect $d\equiv 2 \bmod 4$. Here, we define two class functions $f_1,f_2
\in \mbox{CF}(\bG^F)$ by 
\begin{align*}
f_1&:=\tfrac12(\rho_{(014,2)}-\rho_{(012,4)}-\rho_{(024,1)}+\rho_{(124,0)}), \\
f_2&:=\tfrac12(\rho_{(0124,13)}-\rho_{(0123,14)}-\rho_{(0134,12)}+\rho_{(1234,01)}).
\end{align*}
As before, we check that $\langle f_1,f_2\rangle_{\bG^F}=0$ and $\langle f_i,R_{\bT,1}^\bG
\rangle_{\bG^F}=0$ for $i=1,2$ and for all $\bT$. So again, every unipotent 
character of $\bG^F$ is uniquely a linear combination of $f_1$, $f_2$
and the various $R_{\bT,1}^\bG$.
\end{abs}

The functions $f_i$ ($i=1,2$) are examples of so-called (unipotent) ``almost characters'' as defined in \cite[4.24]{LuB}, that is, they arise from the unipotent characters by applying Lusztig's non-abelian Fourier transform which we referred to in \ref{lfrom}. Hence, by \cite[Theorem~3.2]{S2}, the $f_i$ are in fact characteristic functions of character sheaves. (In the split case where $\bG^F=D_5(q)$, it would be sufficient to refer to \cite[Cor.~14.14]{L2c} for our purposes; see \ref{Ree} below.)
The values of the $f_i$ on unipotent elements of~$\bG^F$ can thus be evaluated up to multiplication with a root of unity using the generalised Springer
correspondence and the algorithm in \cite[\S 24]{L2d} which we will now briefly describe, using a notation similar to the one in \cite[1.1--1.3]{SGenGreen}. (Both the generalised Springer correspondence and Lusztig's algorithm can be accessed through Michel's {\sf CHEVIE} \cite{gap3jm}.)

\begin{abs}\label{LuAlg}
Let $\cN_\bG$ be the set of all pairs $(\cO,\varsigma)$ where $\cO\subseteq\bG$ is a unipotent class and, for a fixed representative $u_\cO\in\cO$, $\varsigma$ is an irreducible character of the group $A_\bG(u_\cO)$. Since any unipotent class of $\bG$ is $F$-stable, we may (and will) assume that $u_\cO\in\cO^F$. Moreover, the group $A_\bG(u_\cO)$ is either trivial or isomorphic to $\Z/2\Z$ for any unipotent class $\cO\subseteq\bG$; we can therefore denote the irreducible characters of $A_\bG(u_\cO)=\langle\overline u_\cO\rangle$ by their values at $\overline u_\cO$. Regarding the unipotent classes of $\bG$ themselves, we denote them by the partition of $10$ giving the sizes of the Jordan blocks.

Each element $(\cO,\varsigma)\in\cN_\bG$ parametrises an $F$-invariant character sheaf on $\bG$, to which Lusztig \cite[24.2]{L2d} associates a function $X_{(\cO,\varsigma)}\in\CF_{\mathrm{uni}}(\bG^F)$. From the generalised Springer correspondence (defined in \cite{LuIC}), it follows that the character sheaf corresponding to $f_1$ is parametrised by $(82,-1)$ and the character sheaf corresponding to $f_2$ by $(6211,-1)$. (We have $A_\bG(u)\cong \Z/2\Z$ for $u$ an element of any of the classes labelled by $82$ and $6211$.) In general, the functions $X_{(\cO,\varsigma)}$ are only defined up to multiplication with a root of unity. As for the pairs $(82,-1)$ and $(6211,-1)$, we can specify the associated functions by setting
\begin{align*}
X_{(82,-1)}(u)&:=q^{-2}f_1(u), \\
X_{(6211,-1)}(u)&:=q^{-4}f_2(u),
\end{align*}
for $u\in\bG^F$ unipotent. Having fixed such normalisations of these ``$X$-functions'', Lusztig's algorithm expresses them as explicit linear combinations of ``$Y$-functions'', which are non-zero only at elements of a single unipotent class of $\bG$ and which are easy to compute up to multiplication with a root of unity: For any $(\cO,\varsigma)\in\cN_\bG$, there is a function $Y_{(\cO,\varsigma)}\in\CF_{\mathrm{uni}}(\bG^F)$ given by $Y_{(\cO,\varsigma)}=\gamma_{(\cO,\varsigma)}Y_{(\cO,\varsigma)}^0$ for a root of unity $\gamma_{(\cO,\varsigma)}\in\C$, where
\[Y_{(\cO,\varsigma)}^0(u):=\begin{cases}
\hfil \varsigma(a) &\text{if $u$ is $\bG^F$-conjugate to $(u_{\cO})_a$ with $a\in A_\bG(u_{\cO})$},  \\
\hfil0 &\text{if }u\notin\cO^F.
\end{cases}\] 
(Note that $Y_{(\cO,\varsigma)}^0$ and thus $\gamma_{(\cO,\varsigma)}$ depend upon the choice of $u_\cO\in\cO^F$.) Regarding the functions $X_{(82,-1)}$ and $X_{(6211,-1)}$ defined above, we then get 
\begin{align*}
f_1(u)&=q^2X_{(82,-1)}(u)=q^2Y_{(82,-1)}(u)+q^3Y_{(6211,-1)}(u), \\
f_2(u)&=q^4X_{(6211,-1)}(u)=q^4Y_{(6211,-1)}(u),
\end{align*}
for $u\in\bG^F$ unipotent. In particular, we already know that the restriction of the $f_i$ to unipotent elements can only be non-zero at the classes labelled by $82$ and $6211$. In order to compute these missing values, it remains to determine the two roots of unity $\gamma_{(82,-1)}$ and $\gamma_{(6211,-1)}$.
\end{abs}

\begin{abs}\label{Ree}
Let $\cO\subseteq\bG$ be an $F$-stable unipotent class of $\bG$ such that $\cO^F$ splits into the unipotent classes $C_1,\ldots, C_r$ of $\bG^F$, and let the further notation be as in \ref{lfrom}. Then we have the identity
\begin{equation*}
\frac{|\bB^Fw\bB^F\cap C_i|\cdot|C_i|}{|\bG^F:\bB^F|}= \sum_{\phi\in \Irr(\bW^F)} 
\phi_q(T_w)[\phi](u)\qquad (1\leqslant i\leqslant r,\; w\in\bW^F,\; u\in C_i)\tag{$*$}
\end{equation*}
where $\phi_q$ is the irreducible character of the Iwahori--Hecke algebra
$\cH$ associated with $\bW^F$ (and a certain weight function depending on $F$), and $[\phi]$ is the unipotent character of
$\bG^F$ corresponding to $\phi\in\Irr(\bW^F)$. (See, for example, 
\cite[\S 11D]{CR2}.) The values 
$\phi_q(T_w)$ are explicitly known via {\sf CHEVIE} \cite{chevie}. Furthermore, by the discussion in \ref{D5a}, each $[\phi]$ can be written as
\[[\phi]=f_{\mathrm{unif}}(\phi)+a_1(\phi)f_1+a_2(\phi)f_2,\]
where $f_{\mathrm{unif}}(\phi)\in\CF(\bG^F)$ is a uniform function (whose values on unipotent elements can thus be computed) and the coefficients $a_1(\phi), a_2(\phi)\in\C$ are explicitly known. We will evaluate ($*$) with $F$-stable elements $u$ in one of the two classes $82$ and $6211$ (and with suitably chosen $w\in\bW^F$). As already noted in \ref{LuAlg}, we have $A_\bG(u)\cong \Z/2\Z$ for any such $u$; in particular, both of the unipotent classes $82$ and $6211$ of $\bG$ split into two $\bG^F$-classes.
\end{abs}

\begin{abs} \label{82} {\bf The regular unipotent class}. Let us first consider the values of $f_1$ and $f_2$ at elements of the regular unipotent class $82$ of $\bG$. Let $u\in\bG^F$ be an element of this class. As already observed in \ref{LuAlg}, we have $f_2(u)=0$, so it remains to deal with $f_1(u)$.

(a) Assume that $\bG^F=D_5(q)$. For $\phi\in\Irr(\bW)$, we get (with the notation of \ref{LuAlg} and \ref{Ree})
\[[\phi](u)=f_{\mathrm{unif}}(\phi)(u)+a_1(\phi)q^2Y_{(82,-1)}(u).\]
We have $f_{\mathrm{unif}}(\phi)(u)=0$ unless $\phi=1_\bW$ is the trivial character of $\bW$, in which case $f_{\mathrm{unif}}(1_\bW)(u)=1$. For $w\in\bW$, the sum on the right hand side of ($*$) in \ref{Ree} then evaluates to
\[\sum_{\phi\in \Irr(\bW)} \phi_q(T_w)[\phi](u)=(1_\bW)_q(T_w)+\tfrac12q^2Y_{(82,-1)}(u)((1.31)_q(T_w)-(11.3)_q(T_w)-(.32)_q(T_w)).\]
Taking $w=w_{\mathrm c}:=s_1s_2s_3s_4s_5$ to be a Coxeter element of $\bW$ (where for $1\leqslant i\leqslant5$, $s_i$ denotes the reflection in the root $\alpha_i$ with respect to the diagram for $D_5$ printed in Section~\ref{subsp}), we have 
\[(1_\bW)_q(T_{w_{\mathrm c}})=q^5,\quad(1.31)_q(T_{w_{\mathrm c}})=-(.32)_q(T_{w_{\mathrm c}})=q^3\quad\text{and}\quad(11.3)_q(T_{w_{\mathrm c}})=0.\]
Hence, denoting by $C\subseteq\bG^F$ the conjugacy class containing $u$, the equation ($*$) in \ref{Ree} reads
\[\frac{|\bB^Fw_{\mathrm c}\bB^F\cap C|\cdot|C|}{|\bG^F:\bB^F|}=q^5(1+Y_{(82,-1)}(u))=q^5(1+\gamma_{(82,-1)}Y_{(82,-1)}^0(u)).\]
Since the left hand side is in particular a real number (and $Y_{(82,-1)}^0(u)\in\{\pm1\}$), the root of unity $\gamma_{(82,-1)}$ can only be $\pm1$. Now recall that $Y_{(82,-1)}^0$ and $\gamma_{(82,-1)}$ depend on the choice of $u_{(82)}$, which we have not yet fixed.~---~If we replace $u_{(82)}$ by a representative of the other $\bG^F$-class contained in the class $82$ of $\bG$, this changes the sign of $\gamma_{(82,-1)}$. We can therefore \emph{define} $u_{(82)}$ by the condition that $\gamma_{(82,-1)}=+1$. Note that the right hand side of the above identity then equals $2q^5$ if $u$ is $\bG^F$-conjugate to $u_{(82)}$, and it is $0$ if $u$ is not $\bG^F$-conjugate to $u_{(82)}$. Hence, the $\bG^F$-class of our chosen $u_{(82)}$ is characterised by the property that it has a non-empty intersection with $\bB^Fw_\mathrm c\bB^F$.

(b) Assume that $\bG^F={^2\!D}_5(q)$. The argument here is very similar to the one for the split case, the only difference is that the group $\bW^F$ is now of type $B_4$. Indeed, the sum $\sum_{\phi\in \Irr(\bW^F)} \phi_q(T_w)[\phi](u)$ evaluates to
\[(1_{\bW^F})_q(T_w)+\tfrac12q^2Y_{(82,-1)}(u)((2.2)_q(T_w)-(21.1)_q(T_w)-(.4)_q(T_w)+(211.)_q(T_w)).\]
Taking again $w=w_\mathrm c=s_1s_2s_3s_4s_5\in\bW^F$ (note that the automorphism of $\bW$ induced by $F$ interchanges $s_1$ and $s_2$, but we have $s_1s_2=s_2s_1$), we get $(1_{\bW^F})_q(T_{w_{\mathrm c}})=q^5$ and
\[(211.)_q(T_{w_{\mathrm c}})=-(.4)_q(T_{w_{\mathrm c}})=q^3,\quad(2.2)_q(T_{w_{\mathrm c}})=(21.1)_q(T_{w_{\mathrm c}})=0.\]
Hence, denoting by $C$ the $\bG^F$-conjugacy class of $u$, we obtain
\[\frac{|\bB^Fw_{\mathrm c}\bB^F\cap C|\cdot|C|}{|\bG^F:\bB^F|}=q^5(1+\gamma_{(82,-1)}Y_{(82,-1)}^0(u)).\]
As in (a) we see that $u_{(82)}$ can be chosen such that $\gamma_{(82,-1)}=+1$, and this determines the $\bG^F$-class of $u_{(82)}$ (inside the regular unipotent class $82$) uniquely; moreover, the $\bG^F$-class of $u_{(82)}$ is characterised by the property that it meets $\bB^Fw_\mathrm c\bB^F$.
\end{abs}

\begin{abs} \label{6211} {\bf The class 6211}. It remains to find the values of $f_1$ and $f_2$ at elements of the unipotent class of $\bG$ labelled by $6211$. So let $u\in\bG^F$ be an element of this class. Recall that $A_\bG(u)\cong\Z/2\Z$.

(a) Assume that $\bG^F=D_5(q)$. For $\phi\in\Irr(\bW)$, we have
\[[\phi](u)=f_{\mathrm{unif}}(\phi)(u)+(a_1(\phi)q^3+a_2(\phi)q^4)Y_{(6211,-1)}(u)\]
(see \ref{LuAlg}, \ref{Ree}). Now the value $f_{\mathrm{unif}}(\phi)(u)$ is non-zero for $\phi\in\{1_\bW,1.4,.41\}$. In fact, we have $[\phi](u)=f_{\mathrm{unif}}(\phi)(u)$ for any of these $\phi$ and
\[f_{\mathrm{unif}}(1_\bW)(u)=1,\quad f_{\mathrm{unif}}(1.4)(u)=q,\quad f_{\mathrm{unif}}(.41)(u)=q^2.\]
For $w\in\bW$, we again evaluate the sum on the right hand side of ($*$) in \ref{Ree}:
\begin{align*}
\sum_{\phi\in \Irr(\bW)} \phi_q(T_w)[\phi](u)&=(1_\bW)_q(T_w)+(1.4)_q(T_w)\cdot q+(.41)_q(T_w)\cdot q^2 \\
&+\tfrac12q^3Y_{(6211,-1)}(u)((1.31)_q(T_w)-(11.3)_q(T_w)-(.32)_q(T_w)) \\
&+\tfrac12q^4Y_{(6211,-1)}(u)((1.211)_q(T_w)-(111.2)_q(T_w)-(.211)_q(T_w)).
\end{align*}
Here we take $w:=s_1s_2s_3s_4\in\bW$ and get
\begin{align*}
&(1_\bW)_q(T_w)=(.41)_q(T_w)=q^4,\quad(1.4)_q(T_w)=q^4-q^3, \\
&(1.31)_q(T_w)=-q^3+q^2,\quad(11.3)_q(T_w)=-q^3,\quad(.32)_q(T_w)=-q^2,\\
&(1.211)_q(T_w)=q^2-q,\quad(111.2)_q(T_w)=-q,\quad(.221)_q(T_w)=-q^2.
\end{align*}
Hence, if $C\subseteq\bG^F$ is the conjugacy class containing $u$, the equation ($*$) in \ref{Ree} reads
\[\frac{|\bB^Fw\bB^F\cap C|\cdot|C|}{|\bG^F:\bB^F|}=(q^6+q^5)(1+\gamma_{(6211,-1)}Y_{(6211,-1)}^0(u)).\]
We can thus argue as in \ref{82} to deduce that $u_{(6211)}$ can be chosen so that $\gamma_{(6211,-1)}=+1$, and this determines the $\bG^F$-class of $u_{(6211)}$ uniquely; moreover, this $\bG^F$-class is characterised by the property that it has a non-empty intersection with $\bB^Fw\bB^F$.

(b) Assume that $\bG^F={^2\!D}_5(q)$. The argument is again similar to the one in (a). The sum $\sum_{\phi\in \Irr(\bW^F)} \phi_q(T_w)[\phi](u)$ now evaluates to
\begin{align*}
&(1_{\bW^F})_q(T_w)+(31.)_q(T_w)\cdot q+(3.1)_q(T_w)\cdot q^2 \\
&+\tfrac12q^3Y_{(6211,-1)}(u)((2.2)_q(T_w)-(.4)_q(T_w)-(21.1)_q(T_w)+(211.)_q(T_w)) \\
&+\tfrac12q^4Y_{(6211,-1)}(u)((1.21)_q(T_w)-(.31)_q(T_w)-(11.11)_q(T_w)+(1111.)_q(T_w)).
\end{align*}
Taking $w:=s_1s_2s_3s_4\in\bW^F$, we have
\begin{align*}
&(1_{\bW^F})_q(T_w)=q^4,\quad(31.)_q(T_w)=q^4-q^3,\quad(3.1)_q(T_w)=q^4, \\
&(2.2)_q(T_w)=0,\quad(.4)_q(T_w)=-q^2,\quad(21.1)_q(T_w)=-q^3,\quad(211.)_q(T_w)=-q^3+q^2,\\
&(1.21)_q(T_w)=q,\quad(.31)_q(T_w)=-q^2+q,\quad(11.11)_q(T_w)=0,\quad(1111.)_q(T_w)=q^2.
\end{align*}
Hence, denoting by $C$ the $\bG^F$-conjugacy class of $u$, we obtain
\[\frac{|\bB^Fw\bB^F\cap C|\cdot|C|}{|\bG^F:\bB^F|}=(q^6+q^5)(1+\gamma_{(6211,-1)}Y_{(6211,-1)}^0(u)).\]
As before we conclude that $u_{(6211)}$ can be chosen such that $\gamma_{(6211,-1)}=+1$, and this determines the $\bG^F$-class of $u_{(6211)}$ uniquely; moreover, the $\bG^F$-class of $u_{(6211)}$ is characterised by the property that it meets $\bB^Fw\bB^F$.
\end{abs}

\begin{abs} In \ref{82} and \ref{6211}, we defined $u_{(82)}$ and $u_{(6211)}$ by requiring the signs $\gamma_{(82,-1)}$ and $\gamma_{(6211,-1)}$ to be $+1$, which we showed to be equivalent to saying that their $\bG^F$-classes have a non-empty intersection with the $F$-stable points of a suitable Bruhat cell. Hence, if we are given an explicit element $u\in\bG^F$ lying in the class $82$ (respectively, $6211$) of the algebraic group $\bG$, in order to show that $u$ is $\bG^F$-conjugate to $u_{(82)}$ (respectively, $u_{(6211)}$), it suffices to find \emph{one} $\bG^F$-conjugate of $u$ which lies in the associated Bruhat cell. This can be achieved using the criterion in \cite[Lemma~3.2.24]{Het3}. Indeed, let $\bU\subseteq\bB$ be the unipotent radical of the Borel subgroup $\bB$ and, for $1\leqslant i\leqslant5$, let $\bU_{\alpha_i}\subseteq\bU$ be the root subgroup associated to the simple root $\alpha_i$, with corresponding homomorphism $u_i:=u_{\alpha_i}\colon k\rightarrow\bU_{\alpha_i}$. Then $u_1(1)u_2(1)u_3(1)u_4(1)u_5(1)$ is an element of the regular unipotent class $82$ of $\bG$, and it also lies in $\bG^F$ regardless of whether $\bG^F=D_5(q)$ or $\bG^F={^2\!D}_5(q)$ (as $u_1(1)u_2(1)=u_2(1)u_1(1)$). Similarly, $u_1(1)u_2(1)u_3(1)u_4(1)\in\bG^F$ is an element of the unipotent class $6211$ of $\bG$. In view of our definition of the $w\in\bW^F$ associated to the classes $82$ (in \ref{82}) and $6211$ (in \ref{6211}), it directly follows from \cite[Lemma~3.2.24]{Het3} that we may take
\begin{align*}
u_{(82)}&=u_1(1)u_2(1)u_3(1)u_4(1)u_5(1),\\
u_{(6211)}&=u_1(1)u_2(1)u_3(1)u_4(1).
\end{align*}
The values at unipotent elements of the class functions $f_1$ and $f_2$ defined in \ref{D5a} are given by Table~\ref{f1f2}, where $u_{(82)}'\in\bG^F$ denotes an element of the regular unipotent class $82$ which is not $\bG^F$-conjugate to $u_{(82)}$, and $u_{(6211)}'\in\bG^F$ is an element of the unipotent class $6211$ which is not $\bG^F$-conjugate to $u_{(6211)}$.
\begin{table}[htbp] \caption{Values of $f_1$ and $f_2$ at unipotent elements for groups of type $D_5$} \label{f1f2}
\begin{center} $\renewcommand{\arraystretch}{1.1} \begin{array}{lccccccccc} 
\hline & u_{(82)} & u_{(82)}' & u_{(6211)}  & u_{(6211)}' & u\notin (82\cup 6211) \\ \hline 
f_1 & q^2 & -q^2 & q^3 & -q^3 & 0 \\ 
f_2 & 0 &  0 & q^4 & -q^4 & 0 \\ \hline 
\end{array}$
\end{center}
\end{table}
\end{abs}

\bigskip
\noindent {\bf Acknowledgements}. I thank Jonas Hetz and Gunter Malle for a 
careful reading of the manuscript and a number of useful comments. I am
indebted to Jonas Hetz for allowing me to include his results on groups 
of type $D_5$ in an appendix here. Thanks are also due to Jean Michel 
for help with his {\tt Julia} version of the {\sf CHEVIE} package. This 
article is a contribution to SFB-TRR 195 by the DFG (Deutsche
Forschungsgemeinschaft), Project-ID 286237555. 


\end{document}